\documentclass[a4paper, reqno, 12pt]{amsart}
\usepackage[margin=1.25in]{geometry}
\usepackage{graphicx} % Required for inserting images
\usepackage{amsthm,amsfonts,amssymb,amsmath,amsxtra,amsrefs}
\usepackage{mathrsfs}
\usepackage{graphicx}
\usepackage{tikz}
\usepackage{tikz-cd}

\usepackage[new]{old-arrows}

\usepackage{xr-hyper}
\usepackage[colorlinks=false,
citecolor=Black,
linkcolor=Red,
urlcolor=Blue]{hyperref}
\usepackage{verbatim}

\usepackage{todonotes,cancel}

\newcommand{\qbinom}[2]{\genfrac{[}{]}{0pt}{0}{#1}{#2}}

\newcommand{\intU}{\mathbf{U}}
\newcommand{\wt}{\text{wt}}
\newcommand{\fr}{{\textbf{Fr}'}}
\newcommand{\ofr}{\textbf{Fr}}
\newcommand{\intA}{\mathbf{A}}
\newcommand{\dfr}{\fr^{,*}}
\newcommand{\dofr}{\textbf{Fr}^*}
\newcommand{\intV}{\mathbf{V}}
\newcommand{\intI}{\mathbf{I}}

\newcommand{\tB}{{\widetilde{B}}}
\newcommand{\ex}{\text{ex}}
\newcommand{\fz}{\text{fz}}

\newcommand{\kk}{\mathbf{k}}

\newcommand{\cA}{\mathscr{A}}

\newcommand{\cF}{\mathscr{F}}

\newcommand{\cL}{\mathscr{S}}

\newcommand{\bx}{\mathbf{x}}
\newcommand{\be}{\mathbf{e}}
\newcommand{\dt}{\mathrm{t}}

\newtheorem{thm}{Theorem}
\newtheorem{theorem}[thm]{Theorem}

\newtheorem{prop}[thm]{Proposition}

\newtheorem{lemma}[thm]{Lemma}

\newtheorem{cor}[thm]{Corollary}
\newtheorem{corollary}[thm]{Corollary}

\theoremstyle{definition}

\newtheorem{remark}[thm]{Remark}

\title{Quantum Frobenius splittings and cluster structures}
\author[Jinfeng Song]{Jinfeng Song}
\address{Department of Mathematics, National University of Singapore, Singapore.}
\email{j\_song@u.nus.edu}

\begin{document}
	
	
	\maketitle
	
	\begin{abstract}
		We prove that the duals of the quantum Frobenius morphisms and their splittings by Lusztig are compatible with quantum cluster monomials. After specialisation, we deduce that the canonical Frobenius splittings on flag varieties are compatible with cluster algebra structures on Schubert cells. 
	\end{abstract}
	
	\section{Introduction}
	
	Let $\mathfrak{g}=\mathfrak{n}_-\oplus\mathfrak{h}\oplus\mathfrak{n}$ be a triangular decomposition of a symmetrizable Kac-Moody Lie algebra, and $U_q(\mathfrak{n})$ be the quantized universial enveloping algebra associated to $\mathfrak{n}$, with the Lusztig integral form $\intU_q(\mathfrak{n})$ which is the $\mathbb{Z}[q^{\pm1}]$-subalgebra generated by divided powers.
	
	Let $l$ be an odd integer which is coprime to all the root lengths of $\mathfrak{g}$, and $\varepsilon$ be a primitive $l$-th root of unity. Lusztig \cites{Lu90,Lu10} constructed a quantum Frobenius morphism $\ofr:\intU_\varepsilon(\mathfrak{n})\rightarrow\intU_1(\mathfrak{n})$ from the quantized universial enveloping algebra at $l$-th root of unity to the classical universial enveloping algebra and a splitting $\fr:\intU_1(\mathfrak{n})\rightarrow\intU_\varepsilon(\mathfrak{n})$ such that $\ofr\circ\fr=id$. These two maps have important applications in representation theory and geometry. To name a few, the map $\ofr$ plays essential role in the study of representations of reductive groups over positive characteristic \cites{AJS,KL94}, and the map $\fr$ is used to construct Frobenius splittings of Schubert varieties in \cite{KL}.
	
	Lusztig's original definitions of these two maps are only derived through brute force computations. It is desirable to obtain more conceptual understandings for these maps. Several progresses have been made towards this question. McGerty \cite{Mc} gave a Hall algebra construction of the map $\ofr$. Qi \cite{Qi} categorified $\ofr$ in rank one at prime roots of unity via p-DG algebras. The map $\fr$ seems to be more mysterious, and no similar results are known. In the current paper we make another attempt to answer this question from the cluster theory point of view. Our approach treats both of the maps $\ofr$ and $\fr$ simultaneously.
	
	%: The map $\ofr$ is not compatible with canonical basis, in the sense that $\ofr$ does not send canonical basis elements to canonical basis elements or zero.  
	
	Let $\intA_q(\mathfrak{n})$ be the graded dual of the integral form $\intU_q(\mathfrak{n})$. Take any element $w$ in the Weyl group of $\mathfrak{g}$. Assume $l(w)=r$, where $l(\cdot)$ is the length function. Following \cite{Ki}, the (integral) quantized coordinate ring $\intA_q(\mathfrak{n}(w))$ of the unipotent subgroup $N(w)=N\cap w^{-1}N_-w$ is defined as a $\mathbb{Z}[q^{\pm1}]$-subalgebra of $\intA_q(\mathfrak{n})$, which is spanned by a dual PBW-type basis associated to a reduced expression of $w$. It follows from Kang--Kashiwara--Kim--Oh \cite{KKKO17} and Goodearl--Yakimov \cite{GY} that the $\mathbb{Z}[q^{\pm1/2}]$-algebra $\intA_{q^{1/2}}(\mathfrak{n}(w))=\mathbb{Z}[q^{\pm1/2}]\otimes_{\mathbb{Z}[q^{\pm1}]}\intA_q(\mathfrak{n}(w))$ has a \emph{quantum cluster algebra} structure of rank $r$ in the sense of Berenstein--Zelevinsky \cite{BZ05}. For any \emph{cluster} $\bx$ of $\intA_{q^{1/2}}(\mathfrak{n}(w))$ and $\mathbf{a}\in\mathbb{N}^r$, we denote $\bx^{\mathbf{a}}$ to be the corresponding \emph{quantum cluster monomial} (see \S \ref{sec:qca} for definitions).
	
	The $\mathbb{Z}[\varepsilon]$-algebras $\intA_\varepsilon(\mathfrak{n})$ and $\intA_1(\mathfrak{n})$ are defined to be the base change of $\intA_{q^{1/2}}(\mathfrak{n})$ via $q^{1/2}\mapsto \varepsilon^{(l+1)/2}$ and $q^{1/2}\mapsto 1$ respectively. They are naturally isomorphic to the graded duals of $\intU_\varepsilon(\mathfrak{n})$ and $\intU_1(\mathfrak{n})$. By taking duals of the map $\ofr$ and $\fr$, we get $\mathbb{Z}[\varepsilon]$-linear maps $\dofr:\intA_1(\mathfrak{n})\rightarrow\intA_\varepsilon(\mathfrak{n})$ and $\dfr:\intA_\varepsilon(\mathfrak{n})\rightarrow \intA_1(\mathfrak{n})$. For any element $f$ in $\intA_{q^{1/2}}(\mathfrak{n})$, we use $_\varepsilon f$ (resp., $_1 f$) to denote its image in $\intA_\varepsilon(\mathfrak{n})$ (resp., $\intA_1(\mathfrak{n})$). The main result of this paper is the following.
	
	\begin{theorem}\label{thm:main}
		Let $w$ be a Weyl group element of length $r$ and $\mathbf{x}^{\mathbf{a}}$ $(\mathbf{a}\in\mathbb{N}^{r})$ be a quantum cluster monomial in $\intA_{q^{1/2}}(\mathfrak{n}(w))$. We have
		\[
		\dofr({_1\mathbf{x}^\mathbf{a}})={_\varepsilon \mathbf{x}^{l\mathbf{a}}}\qquad\text{and}\qquad \dfr({_\varepsilon\mathbf{x}^{\mathbf{a}}})={_1\mathbf{x}^{\mathbf{a}/l}}.
		\]
		Here $\mathbf{x}^{\mathbf{a}/l}$ is understood as $0$ if $\mathbf{a}\not\in l\mathbb{N}^{r}$.
	\end{theorem}
	
	Theorem \ref{thm:main} is proved in \S \ref{sec:proof}. It is noteworthy that our result is related to the question on the compatibility between canonical basis and quantum Frobenius.
	
	The \emph{canonical basis} by Lusztig and Kashiwara is a particular linear basis of $\intU_q(\mathfrak{n})$. In \cite{Mc}*{Remark 5.10}, McGerty raised a question on the compatibility between maps $\ofr$ and $\fr$ with canonical basis. It is later shown by Baumann \cite{Bau} that in general they are not compatible in the strict sense, meaning that these maps do not send a basis element to another basis element or zero, but they are compatible up to certain filtrations. 
	
	The canonical basis of $\intU_q(\mathfrak{n})$ gives the \emph{dual canonical basis} of its graded dual $\intA_q(\mathfrak{n})$. By Kang--Kashiwara--Kim--Oh \cite{KKKO17} and Qin \cite{Qin}, cluster monomials form a subset of dual canonical basis up to scalar. Therefore we expect that our result can provide new insights for McGerty's question.
	
	As an application of Theorem \ref{thm:main}, we show the maps $\dofr$ and $\dfr$ preserve the quantized coordinate ring of the unipotent subgroup associated to the same Weyl group element (Corollary \ref{cor:rest}). To be more precise, for any $w$ in the Weyl group, we obtain the following two inclusions, 
	\begin{equation*}
	\dofr\big(\intA_1(\mathfrak{n}(w))\big)\subset \intA_\varepsilon(\mathfrak{n}(w))\qquad\text{and}\qquad   \dfr\big(\intA_\varepsilon(\mathfrak{n}(w))\big)\subset \intA_1(\mathfrak{n}(w)).
	\end{equation*}
	
	Here $\intA_\varepsilon(\mathfrak{n}(w))$ and $\intA_1(\mathfrak{n}(w))$ are $\mathbb{Z}[\varepsilon]$-subalgebras of $\intA_\varepsilon(\mathfrak{n})$ and $\intA_1(\mathfrak{n})$ respectively, which are spanned by images of dual PBW-type basis under corresponding base changes. The first inclusion also follows from the compatibility between $\ofr$ and braid group actions, when $\mathfrak{g}$ is of finite type. The second inclusion is more subtle, since there is no obvious compatibility between $\fr$ and braid group actions.
	
	We next discuss some geometric applications of our result. Assume $\mathfrak{g}$ is of finite type for the rest of this section. Let $\kk$ be an algebraically closed field of characteristic $l$. Let $G$ be the reductive group defined over $\kk$ associated to the Lie algebra $\mathfrak{g}$, $B$ be the standard Borel subgroup and $G/B$ be the associated flag variety. For elements $w$ and $v$ in the Weyl group of $G$, we denote $C_w=BwB/B$ to be the \emph{Schubert cell}, $C^v=B_-vB/B$ to be the \emph{opposite Schubert cell} and $C_w^v=C_w\cap C^v$ to be the \emph{open Richardson variety}. 
	
	Kumar--Littelmann \cite{KL} showed that the map $\dfr$ provides a quantum lift of a \emph{Frobenius splitting} of the flag variety $G/B$. By \cite{KL}*{Theorem 6.4} and \cite{BK}*{Theorem 4.1.15}, this splitting is the unique \emph{$B$-canonical} splitting of $G/B$ and moreover the splitting compatibly splits all the \emph{Schubert variety} and \emph{opposite Schubert variety}. Therefore the canonical splitting of $G/B$ induces canonical splittings of Schubert cells as well as open Richardson varieties.
	
	One can identify the unipotent subgroup $N(w)$ with the Schubert cell $C_{w^{-1}}$ via the map $n\mapsto n\cdot w^{-1}B/B$. Under this isomorphism the $\kk$-algebra $\intA_\kk(\mathfrak{n}(w))$ obtained from $\intA_{q^{1/2}}(\mathfrak{n}(w))$ by specialising $q^{1/2}$ at 1 is canonically isomorphic to the coordinate ring of the Schubert cell $C_{w^{-1}}$ \cite{Ki}*{Theorem 4.44}. One can show that $\intA_\kk(\mathfrak{n}(w))$ carries a (classical) cluster algebra structure via specialisation, by applying similar arguments as in \cite{GLS20}. Our result implies the compatibility between the canonical splittings on Schubert cells and their cluster algebra structures.
	
	\begin{cor}[Proposition \ref{prop:frclu} \& Corollary \ref{cor:alggeo}]\label{cor:2}\
		\begin{enumerate}
			\item The restriction of $\dfr$ provides a quantum lift of the canonical splitting of the Schubert cell.
			\item The canonical splittings of Schubert cells are compatible with their cluster algebra structures, that is, they divides the degrees by $l$ when acting on any cluster monomials.
		\end{enumerate}
	\end{cor}
	
	%The concept of \emph{Frobenius splitting} was originally introduced by Mehta-Ramanathan \cite{MR85} and Ramanan-Ramanathan \cite{RR85}, in their study of Schubert varieties. We refer to the textbook \cite{BK} for a detailed discussion of this topic. 
	
	We remark that Benito--Muller--Rajchgot--Smith showed in \cite{BMRS15}*{Theorem 3.7} that any \emph{upper cluster algebra} (over $\kk$) admits a unique Frobenius splitting which divides the degrees by $l$ when acting on any cluster monomials. We call it the \emph{cluster splitting}. The associated upper cluster algebra of $\intA_\kk(\mathfrak{n}(w))$ is canonically isomorphic to the coordinate ring of the open Richardson variety $C_{w^{-1}}^e$. Then another way to state Corollary \ref{cor:2} (2) is that the canonical splitting of the open Richardson variety of the form $C_w^e$ ($w\in W$) coincides with its cluster splitting.
	
	The (upper) cluster algebra structures on general open Richardson varieties are recently obtained by \cite{CCGGLSS} and \cite{GLSB}. We expect that the same statement remains true for any open Richardson varieties.
	
	As another geometric application of Theorem \ref{thm:main}, we show that the canonical splitting of the Schubert cell is compatible with reduction maps.
	
	Suppose $w=v'v$ in the Weyl group of $\mathfrak{g}$ such that $l(w)=l(v')+l(v)$. Following  \cite{MR}*{\S 4.3}, define the \emph{reduction map} between Schubert cells,
	\[
	\pi_v^w:C_{w^{-1}}\rightarrow C_{v^{-1}},\qquad b\cdot w^{-1}B/B\mapsto b\cdot v^{-1}B/B.
	\]

	% \[
	% Bw^{-1}B/B\longrightarrow Bv^{-1}B/B\qquad \text{given by $n\cdot w^{-1}B/B\mapsto n\cdot v^{-1}B/B$}.
	% \]
	
	\begin{cor}\label{cor:red}
		The canonical splittings on Schubert cells are compatible with reduction maps, that is, we have the following commuting diagram,
		\begin{equation*}
		\begin{tikzcd}
		&\intA_\kk(\mathfrak{n}(v))\arrow[r,"(\pi_v^w)^*"] \arrow[d,"\varphi_{v}"'] &\intA_\kk(\mathfrak{n}(w)) \arrow[d,"\varphi_{w}"]\\& \intA_\kk(\mathfrak{n}(v))\arrow[r,"(\pi_v^w)^*"]& \intA_\kk(\mathfrak{n}(w)).	
		\end{tikzcd}
		\end{equation*}
		Here $(\pi_v^w)^*$ is the comorphism, and the map $\varphi_v$ (resp., $\varphi_w$) is the canonical splitting of the Schubert cell $C_{v^{-1}}$ (resp., $C_{w^{-1}}$).
	\end{cor}
	Corollary \ref{cor:red} is proved in \S \ref{sec:red}.
	
	The paper is organised as follows. In \S 2 we recall relevant results and prove the main theorem. In \S 3 we recall the concept of Frobenius splittings and deduce geometric consequences by specialising $q$ at 1. 
	
	\noindent {\bf Acknowledgement: }The author would like to thank his supervisor Huanchen Bao for many helpful discussions. The author is supported by Huanchen Bao’s MOE grant A-0004586-00-00 and A-0004586-01-00. 
	
	\section{Quantum Frobenius splittings}\label{sec:sec2}
	
	\subsection{Quantum cluster algebras}\label{sec:qca}
	
	In this subsection, we recall quantum cluster algebras following \cite{BZ05} and \cite{KKKO17}*{\S 5}.
	
	Fix a finite index set $J=J_\ex\sqcup J_\fz$ with the decomposition into exchangeable indices set $J_\ex$ and frozen indices set $J_\fz$. Let $\cF$ be a skew field over $\mathbb{Q}(q^{1/2})$. Here $q^{1/2}$ is an indeterminate.
	
	Let $\Lambda=(\lambda_{ij})_{i,j\in J}$ be a skew-symmetric integer valued $J\times J$-matrix. By abuse of notations, we will also denote by $\Lambda$ the skew-symmetric bilinear form on $\mathbb{Z}^J$ given by $\Lambda(\be_i,\be_j)=\lambda_{ij}$. Here $\{\be_i\mid i\in J\}$ is standard basis of $\mathbb{Z}^J$. A tuple $\bx=(x_i)_{i\in J}$ consisting of elements in $\cF$ is called \emph{$\Lambda$-commutative} if $x_ix_j=q^{\lambda_{ij}}x_jx_i$, for $i,j\in J$. For such a tuple $\bx$ and $\mathbf{a}=(a_i)_{i\in J}\in \mathbb{Z}^J$, define the element
	\begin{equation}\label{eq:twpro}
	\bx^\mathbf{a}=q^{1/2\sum_{i>j}a_ia_j\lambda_{ij}}\prod_{i\in J}x_i^{a_i}\qquad \text{in $\cF$.}
	\end{equation}
	Here we take an arbitrary total order on $J$, and the product is taken increasingly. One can check that this product is independent of the choice of the total order. For any $\mathbf{a},\mathbf{b}\in \mathbb{Z}^J$, by direct computations one has
	\begin{equation}\label{eq:qprod}
	\bx^\mathbf{a}\bx^\mathbf{b}=q^{1/2\Lambda(\mathbf{a},\mathbf{b})}\bx^{\mathbf{a}+\mathbf{b}}\qquad \text{and} \qquad \bx^{\mathbf{a}}\bx^\mathbf{b}=q^{\Lambda(\mathbf{a},\mathbf{b})}\bx^{\mathbf{b}}\bx^{\mathbf{a}}.
	\end{equation}
	
	A $\Lambda$-commutative tuple $\bx=(x_i)_{i\in J}$ is called \emph{algebraically independent} if elements $\bx^{\mathbf{a}}$ ($\mathbf{a}\in\mathbb{Z}^J$) are linearly independent in $\cF$.
	
	Let $\tB=(b_{ij})_{(i,j)\in J\times J_\ex}$ be an integer-valued $J\times J_\ex$ matrix. The pair $(\Lambda,\tB)$ is called \emph{compatible} if for any $j\in J_\ex$ and $i\in J$, we have
	\begin{equation}\label{eq:comp}
	\sum_{k\in J} b_{kj}\lambda_{ki}=\delta_{ij}d_j,
	\end{equation}
	for some positive integer $d_j$. This is equivalent to the requirement that the matrix $D=\tB^T\Lambda$ consists of two blocks: the $J_\ex\times J_\ex$ diagonal matrix with positive diagonal entries $d_j$ ($j\in J_\ex$), and the $J_\ex\times J_\fz$ zero matrix. The matrix $D$ is called the \emph{skew-symmetrizer} of $(\Lambda,\tB)$.
	
	A \emph{quantum seed} $\cL=(\bx,\Lambda,\tB)$ in $\cF$ consists of a compatible pair $(\Lambda,\tB)$ and a $\Lambda$-commutative algebraically independent tuple $\bx=(x_i)_{i\in J}$ in $\cF$. The tuple $\bx$ is called the \emph{cluster} of the quantum seed $\cL$, and its entries are called \emph{cluster variables}. Elements in $\{x_i\}_{i\in J_\ex}$ are called \emph{exchangeable variables}, while elements in $\{x_i\}_{i\in J_\fz}$ are called \emph{frozen variables}.
	
	Let $\cL=(\bx=(x_i)_{i\in J},\Lambda,\tB)$ be a quantum seed in $\cF$ and $k\in J_\ex$ be an exchangeable index. Define matrices $E=(e_{ij})_{i,j\in J}$ and $F=(f_{ij})_{i,j\in J_\ex}$ as follows,
	\begin{equation*}
	e_{ij}=\left\{\begin{array}{ll}
	\delta_{ij}, & \text{if }j\neq k, \\
	-1, & \text{if }i=j=k, \\
	{[-b_{ik}]_+}, & \text{if }i\neq j=k,
	\end{array}\right.\quad
	f_{ij}=\left\{\begin{array}{ll}
	\delta_{ij}, & \text{if }i\neq k, \\
	-1, & \text{if }i=j=k,\\
	{[b_{kj}]_+}, & \text{if }i=k\neq j.
	\end{array}\right.
	\end{equation*} 
	Here we write $[n]_+=\text{max}(n,0)$ for any $n\in\mathbb{Z}$. Set $$\mu_k(\tB)=E\tB F \quad\text{and}\quad \mu_k(\Lambda)=E^T\Lambda E.$$ Then the pair $\mu_k(\Lambda,\tB)=(\mu_k(\Lambda),\mu_k(\tB))$ is again compatible with the same skew-symmetrizer $D$ (\cite{BZ05}*{Proposition 3.4}). It is called the \emph{mutation of $(\Lambda,\tB)$ in the direction $k$}.
	
	Let $\mathbf{b}_k=\sum_{i\in J}b_{ik}\be_i\in\mathbb{Z}^J$ be the $k$-th column vector of $\tB$. Set
	\[
	\mu_k(x_i)=\left\{\begin{array}{ll}
	\bx^{-\be_k+[\mathbf{b}_k]_+}+\bx^{-\be_k+[-\mathbf{b}_k]_+}, & \text{if }i=k, \\
	x_i, & \text{if }i\neq k.
	\end{array}\right.
	\]
	Here we write $[\mathbf{n}]_+=([n_i]_+)_{i\in J}$ for any $\mathbf{n}=(n_i)_{i\in J}\in \mathbb{Z}^J$. Then one can show that the tuple $\mu_k(\bx)=(\mu_k(x_i))_{i\in J}$ is $\mu_k(\Lambda)$-commutative and algebraically independent. Hence $\mu_k(\cL)=(\mu_k(\bx),\mu_k(\Lambda),\mu_k(\tB))$ is also a quantum seed in $\cF$, called the \emph{mutation of $\cL$ in the direction $k$}. Two quantum seeds $\cL'$ and $\cL''$ are called \emph{mutation-equivalent}, denoted by $\cL'\sim\cL''$, if they can be obtained from another by a sequence of mutations.
	
	Given a quantum seed $\cL$ in $\cF$, The \emph{quantum cluster algebra} $\cA_q(\cL)$ is the $\mathbb{Z}[q^{\pm1/2}]$-subalgebra of $\cF$ generated by the union of clusters of all quantum seeds which are mutation-equivalent to $\cL$. The seed $\cL$ is called the \emph{initial seed} of $\cA_q(\cL)$.
	
	Take any quantum seed $\cL'\sim\cL$. Let $\bx'=(x'_i)_{i\in J}$ be the cluster of $\cL'$. The \emph{quantum Laurent phenomenon} (\cite{BZ05}*{Corollary 5.2}) asserts that the cluster algebra $\cA_q(\cL)$ is contained in the $\mathbb{Z}[q^{\pm1/2}]$-subalgebra of $\cF$ generated by $(x_i')^{\pm1}$ ($i\in J_\ex$) and $x'_i$ ($i\in J_\fz$). Elements of the form $(\bx')^{\mathbf{a}}$ ($\mathbf{a}\in\mathbb{N}^J$) are called \emph{cluster monomials} in $\cA_q(\cL)$.
	
	\subsection{Quantum groups}
	
	In this subsection we recall basic construction in quantum groups following \cite{Lu10} and \cite{KKKO17}*{\S 1}.
	
	We fix once for all an index set $I$ and a \emph{Kac-Moody root datum} $(A,P,\Pi,P^\vee,\Pi^\vee)$ which consists of 
	
	(a) a \emph{symmetrizable generalised Cartan matrix} $A=(a_{ij})_{i,j\in I}$, that is, an integer valued matrix such that: (i) $a_{ii}=2$, for all $i\in I$, (ii) $a_{ij}<0$, for all $i,j\in I$, (iii) there exists a diagonal matrix $D=\text{diag}(\dt_i\mid i\in I)$ with positive integers $\dt_i$ such that $DA$ is symmetric. We normalize $(\dt_i)_{i\in I}$ such that they are relatively prime.
	
	(b) a finitely generated free abelian group $P$, called the \emph{weight lattice}.
	
	(c) a set of linearly independent elements $\Pi=\{\alpha_i\mid i\in I\}\subset P$, called the set of \emph{simple roots}.
	
	(d) the dual group $P^\vee=\text{Hom}_\mathbb{Z}(P,\mathbb{Z})$, called the \emph{coweight lattice}.
	
	(e) a set of elements $\Pi^\vee=\{h_i\mid i\in I\}\subset P^\vee$, called the set of \emph{simple coroots}, such that $\langle h_i,\alpha_j\rangle=a_{ij}$, for $i,j\in I$.
	
	We further assume the datum to be \emph{simply-connected}, that is, for each $i\in I$, there exits $\varpi_i\in P$ such that $\langle h_j,\varpi_i\rangle=\delta_{ij}$ for all $j\in I$. The element $\varpi_i$ is called the \emph{fundamental weight} associated to $i$.
	
	Set $P_+=\{\mu\in P\mid \langle h_i,\mu\rangle \geq 0, \forall i\in I\}$ to be the set of \emph{dominant weights}. The free abelian group $Q=\bigoplus_{i\in I}\mathbb{Z}\alpha_i$ is called the \emph{root lattice}. Let $\Delta_+$ be the set of simple roots, and $\Delta_-=-\Delta_+$. Set $Q_+=\bigoplus_{i\in I}\mathbb{Z}_{\geq 0}\alpha_i$, and $Q_-=\bigoplus_{i\in I}\mathbb{Z}_{\leq 0}\alpha_i$. There is a symmetric bilinear form $(\cdot,\cdot)$ on $\mathbb{Q}\otimes_\mathbb{Z} P$ such that  $$(\alpha_i,\alpha_j)=\dt_ia_{ij}\quad \text{and}\quad \langle h_i,\lambda\rangle=\frac{2(\alpha_i,\lambda)}{(\alpha_i,\alpha_i)}\qquad \text{for } \lambda \in P \text{ and }i,j\in I.$$
	
	Let $\mathfrak{g}$ be the symmetrizable Kac-Moody algebra over $\mathbb{Q}$ associated to the root datum, and $W$ be the Weyl group of $\mathfrak{g}$ with generators $s_i$ ($i\in I$). Let $l(\cdot)$ be the length function of $W$. A \emph{reduced expression} of $w$ is a sequence $(i_1,\dots,i_{r})\in I^{r}$ such that $w=s_{i_1}\dots s_{i_r}$ in $W$ and $r=l(w)$.
	
	The \emph{quantum group} $U_q(\mathfrak{g})$ associated to the root datum is the unital $\mathbb{Q}(q)$-algebra with generators $e_i$ ($i\in I$), $f_i$ ($i\in I$), and $q^h$ ($h\in P^\vee$), which subject to the following relations
	\begin{align*}
	&q^0=1,\quad q^hq^{h'}=q^{h+h'}, \quad  q^he_iq^{-h}=q^{\langle h,\alpha_i\rangle}e_i,\quad q^hf_iq^{-h}=q^{-\langle h,\alpha_i\rangle},\\& e_if_j-f_je_i=\delta_{ij}\frac{q^{\dt_ih_i}-q^{-\dt_ih_i}}{q_i-q_i^{-1}}, \\
	&\sum_{s=0}^{1-a_{ij}}(-1)^s\qbinom{1-a_{ij}}{s}_ie_i^{1-a_{ij}-s}e_je_i^s=0,\quad \sum_{s=0}^{1-a_{ij}}(-1)^s\qbinom{1-a_{ij}}{s}_if_i^{1-a_{ij}-s}f_jf_i^s=0,
	\end{align*}
	for $h,h'\in P^\vee$ and $i,j\in I$ with $i\neq j$. 
	
	Here we set
	\[
	q_i=q^{\dt_i},\quad [n]_i=\frac{q_i^n-q_i^{-n}}{q_i-q_i^{-1}},\quad [n]_i!=[1]_i\dots[n]_i\quad \text{and } \qbinom{n}{s}_i=\frac{[n]_i!}{[n-s]_i![s]_i!}.
	\]
	
	Let $U_q(\mathfrak{n})$ (resp., $U_q(\mathfrak{n}_-)$) be the subalgebra of $U_q(\mathfrak{g})$ generated by $e_i$'s (resp., $f_i$'s). For any $i\in I$, and $n\in\mathbb{N}$, define \emph{divided powers}, $$e_i^{(n)}=\frac{e_i^n}{[n]_i!}\qquad\text{and}\qquad f_i^{(n)}=\frac{f_i^n}{[n]_i!}.$$
	The \emph{integral form} $\intU_q(\mathfrak{n})$ (resp., $\intU_q(\mathfrak{n}_-)$) is the $\mathbb{Z}[q^{\pm1}]$-subalgebra of $U_q(\mathfrak{n})$ (resp., $U_q(\mathfrak{n}_-)$) generated by $e_i^{(n)}$ (resp., $f_i^{(n)}$) for all $i\in I$ and $n\in \mathbb{N}$. 
	
	There is a $\mathbb{Q}(q)$-algebra anti-automorphism $\phi$ on $U_q(\mathfrak{g})$ such that $$\phi(e_i)=f_i,\quad \phi(f_i)=e_i,\quad\text{and}\quad \phi(q^h)=q^h,$$ for $i\in I$ and $h\in P^\vee$. It is clear that $\phi\big(\intU_q(\mathfrak{n})\big)=\intU_q(\mathfrak{n}_-).$
	
	Note that $U_q(\mathfrak{n})$ has a natural $Q_+$-grading by setting $e_i$ to be of degree $\alpha_i$ ($i\in I$). The degree of a homogeneous element $x$ in $U_q(\mathfrak{n})$ will be called the \emph{weight} of $x$, denoted by $\wt(x)$. For $\gamma\in Q_+$, we write $U_q(\mathfrak{n})_\gamma$ to be the subspace of $U_q(\mathfrak{n})$ consisting of homogeneous elements of weight $\gamma$.
	
	Let $\mathbf{B}$ be the canonical basis of $U_q(\mathfrak{n})$ (\cite{Lu10}*{14.4}). For $\lambda\in Q_+$, set $\mathbf{B}_\lambda\subset \mathbf{B}$ to be the subset consisting of elements of weight $\lambda$. 
	
	Following \cite{Lu10}*{1.2.10}, we endow the tensor product $U_q(\mathfrak{n})\otimes_{\mathbb{Q}(q)} U_q(\mathfrak{n})$ with the (twisted) algebra structure given by
	\[
	(x_1\otimes x_2)\cdot (y_1\otimes y_2)=q^{-(\wt(x_2),\wt(y_1))}(x_1y_1\otimes x_2y_2),
	\]
	for homogeneous elements $x_1,x_2,y_1,y_2$ in $U_q(\mathfrak{n})$. One can identify the tensor product $\intU_q(\mathfrak{n})\otimes_{\mathbb{Z}[q^{\pm1}]}\intU_q(\mathfrak{n})$ as a $\mathbb{Z}[q^{\pm1}]$-subspace of $U_q(\mathfrak{n})\otimes_{\mathbb{Q}(q)} U_q(\mathfrak{n})$, which is moreover a $\mathbb{Z}[q^{\pm1}]$-subalgebra under the twisted algebra structure.
	
	Let $r:U_q(\mathfrak{n})\rightarrow U_q(\mathfrak{n})\otimes_{\mathbb{Q}(q)} U_q(\mathfrak{n})$ be the algebra homomorphism such that
	\[
	r(e_i)=e_i\otimes 1+1\otimes e_i,\quad \text{for $i\in I$}.
	\]
	The map $r$ is called the \emph{coproduct} on $U_q(\mathfrak{n})$.
	One can show (\cite{Lu10}*{1.4.2}) that 
	\begin{equation}\label{eq:cop}
	r(e_i^{(n)})=\sum_{s=0}^nq_i^{-s(n-s)}e_i^{(s)}\otimes e_i^{(n-s)},
	\end{equation}
	for any $i\in I$ and $n\in\mathbb{N}$. Then we have
	\begin{equation}\label{eq:rres}
	r\big(\intU_q(\mathfrak{n})\big)\subset\intU_q(\mathfrak{n})\otimes_{\mathbb{Z}[q^{\pm1}]} \intU_q(\mathfrak{n}).
	\end{equation}
	
	Define the \emph{graded dual} of $U_q(\mathfrak{n})$ as following,
	\[
	A_q(\mathfrak{n})=\bigoplus_{\gamma\in Q^+}\text{Hom}_{\mathbb{Q}(q)}(U_q(\mathfrak{n})_\gamma,\mathbb{Q}(q)).
	\]
	We endow $A_q(\mathfrak{n})$ with the $\mathbb{Q}(q)$-algebra structure given by 
	\[
	(\varphi\cdot \psi)(x)=(\varphi\otimes\psi)(r(x)),
	\]
	for $\varphi,\psi\in A_q(\mathfrak{n})$ and $x\in U_q(\mathfrak{n})$. This algebra structure is well-defined by weight consideration.
	
	% \begin{remark}
	% The definition of the product structure follows from \cite{KKKO17}, which differs from the definition in \cite{GLS} by a power of $q$. (\cite{KKKO17}*{Remark 8.2.3})
	% \end{remark}
	
	Define
	\[
	\intA_q(\mathfrak{n})=\{f\in A_q(\mathfrak{n})\mid f\big(\intU_q(\mathfrak{n})\big)\subset \mathbb{Z}[q^{\pm1}]\}
	\]
	to be the \emph{dual integral form}. Thanks to \eqref{eq:rres}, we deduce that $\intA_q(\mathfrak{n})$ is a $\mathbb{Z}[q^{\pm1}]$-subalgebra of $A_q(\mathfrak{n})$.
	
	By \cite{Lu10}*{1.2.3 \& 1.2.5} there is a unique symmetric nondegenerate biliear form $(\cdot,\cdot):U_q(\mathfrak{n})\times U_q(\mathfrak{n})\rightarrow \mathbb{Q}(q)$ such that
	
	(a) $(1,1)=1$, $(e_i,e_j)=\delta_{ij}(1-q_i^2)^{-1}$, for all $i,j\in I$;
	
	(b) $(x,yy')=(r(x),y\otimes y')$ and $(xx',y)=(x\otimes x',r(y))$, for all $x,x',y,y'\in U_q(\mathfrak{n})$,
	where the bilinear form on $U_q(\mathfrak{n})\otimes U_q(\mathfrak{n})$ is defined by $(x_1\otimes x_2,y_1\otimes y_2)=(x_1,y_1)(y_1,y_2)$.
	
	Then one has the algebra isomorphism 
	\[
	\iota: U_q(\mathfrak{n})\xrightarrow{\sim} A_q(\mathfrak{n}),
	\]
	given by $\iota(x)(y)=(x,y)$ for $x,y\in U_q(\mathfrak{n})$.
	
	\subsection{Quantized coordinate rings of unipotent subgroups}\label{sec:PBW}
	
	We denote by $T_i$ ($i\in I$) the braid group action $T_{i,-1}'$ on $U_q(\mathfrak{g})$ as in \cite{Lu10}*{37.1.3}.
	
	%Write $T_i$ ($i\in I$) to denote the braid group actions on $U_q(\mathfrak{g})$. It is denoted by $T_{i,-1}'$ in \cite{Lu10}*{37.1.3}. 
	
	Let $w\in W$. Fix a reduced expression $\underline{w}=(i_r,\dots, i_1)$ of $w$. For any $1\leq k\leq r$, define elements $$\beta_k=s_{i_1}\dots s_{i_{k-1}}(\alpha_{i_k})\in Q_+ \quad \text{and}\quad e_{\beta_k}=T_{s_{i_1}}\dots T_{s_{i_{k-1}}}(e_{i_k})\in U_q(\mathfrak{n})_{\beta_k}.$$
	
	For any $\mathbf{n}=(n_1,\dots,n_r)\in \mathbb{N}^r$, define elements
	\[
	e_{\underline{w}}(\mathbf{n})=e_{\beta_1}^{(n_1)}\dots e_{\beta_r}^{(n_r)}\in U_q(\mathfrak{n}),
	\]
	where $e_{\beta_k}^{(n_k)}=e_{\beta_k}^{n_k}/[n_k]_{i_k}!$ ($1\leq k\leq r$). Also define elements
	\[
	e_{\underline{w}}'(\mathbf{n})=\frac{e_{\underline{w}}(\mathbf{n})}{(e_{\underline{w}}(\mathbf{n}),e_{\underline{w}}(\mathbf{n}))}\in U_q(\mathfrak{n})\qquad \text{and}\qquad e^*_{\underline{w}}(\mathbf{n})=\iota(e_{\underline{w}}'(\mathbf{n}))\in A_q(\mathfrak{n}).
	\]
	
	Let $U_q(\mathfrak{n}(w))$ be the $\mathbb{Q}(q)$-subspace of $U_q(\mathfrak{n})$ spanned by elements $e_{\underline{w}}(\mathbf{n})$ ($\mathbf{n}\in\mathbb{N}^r$). It follows from \cite{Lu10}*{Proposition 40.2.1} and \cite{Ki}*{Proposition 4.11} that $U_q(\mathfrak{n}(w))$ is a $\mathbb{Q}(q)$-subalgebra of $U_q(\mathfrak{n})$ which is independent of the choice of the reduced expressions of $w$. Set $\intU_q(\mathfrak{n}(w))=U_q(\mathfrak{n}(w))\cap \intU_q(\mathfrak{n})$ to be the integral form. Then $\{e_{\underline{w}}(\mathbf{n})\mid \mathbf{n}\in\mathbb{N}^r\}$ is a $\mathbb{Z}[q^{\pm1}]$-basis of $\intU_q(\mathfrak{n}(w))$, called a \emph{PBW basis} associated to the reduced expression $\underline{w}$.
	
	Define
	$A_q(\mathfrak{n}(w))=\iota\big(U_q(\mathfrak{n}(w))\big)$ to be the quantized coordinate ring of the unipotent subgroup associated with the finite set $\Delta_+\cap w^{-1}\Delta_-$. Let $\intA_q(\mathfrak{n}(w))=\intA_q(\mathfrak{n})\cap A_q(\mathfrak{n}(w))$ be the integral form. The set $\{e^*_{\underline{w}}(\mathbf{n})\mid \mathbf{n}\in\mathbb{N}^r\}$ gives a $\mathbb{Z}[q^{\pm1}]$-basis of $\intA_q(\mathfrak{n}(w))$, called a \emph{dual PBW basis} associated to $\underline{w}$. 
	
	%Then $A_q(\mathfrak{n}(w))$ is a $\mathbb{Q}(q)$-subalgebra of $A_q(\mathfrak{n})$, and $\intA_q(\mathfrak{n}(w))$ is a $\mathbb{Z}[q^{\pm 1}]$-subalgebra of $A_q(\mathfrak{n}(w))$. 
	
	By \cite{Ki17}*{Theorem 2.19}, one has $$U_q(\mathfrak{n}(w))=U_q(\mathfrak{n})\cap T_{w^{-1}}U_q(\mathfrak{n}_-).$$ Here $T_{w^{-1}}=T_{i_1}\dots T_{i_r}$. Also define $$U_q(\mathfrak{n}(w)')=U_q(\mathfrak{n})\cap T_{w^{-1}}U_q(\mathfrak{n})\qquad \text{and}\qquad A_q(\mathfrak{n}(w)')=\iota\big(U_q(\mathfrak{n}(w)')\big).$$
	The algebra $A_q(\mathfrak{n}(w)')$ is called the quantized coordinate ring of the pro-unipotent subgroup associated with the co-finite subset $\Delta_+\cap w^{-1}\Delta_+$.
	Set $$\intA_q(\mathfrak{n}(w)')=A_q(\mathfrak{n}(w)')\cap \intA_q(\mathfrak{n})$$ to be its integral form. Then $\intA_q(\mathfrak{n}(w)')$ is a free $\mathbb{Z}[q^{\pm1}]$-module (\cite{Ki17}*{Theorem 3.3}). Moreover, by \cite{Ki17}*{Theorem 1.1 (2)}, multiplication gives the isomorphism as $\mathbb{Z}[q^{\pm1}]$-modules:
	\begin{equation}\label{eq:tens}
	\intA_q(\mathfrak{n}(w))\otimes_{\mathbb{Z}[q^{\pm1}]}\intA_q(\mathfrak{n}(w)')\xrightarrow{\sim}\intA_q(\mathfrak{n}).
	\end{equation}

	\subsection{Unipotent quantum minors }\label{sec:minor}
	
	For any $\lambda\in P_+$, let $V(\lambda)$ be the integrable simple $U_q(\mathfrak{g})$-module with highest weight $\lambda$ (\cite{KKKO17}*{1.2}). We fix a nonzero vector $v_\lambda\in V(\lambda)$ of weight $\lambda$, and set $\intV(\lambda)=\intU_q(\mathfrak{n}_-)v_\lambda$ to be the integral form of $V(\lambda)$. Then $\intV(\lambda)$ is a free $\mathbb{Z}[q^{\pm1}]$-submodule of $V(\lambda)$.
	
	For $\lambda\in P_+$, define $\intI(\lambda)$ to be the right ideal of $\intU_q(\mathfrak{n})$ given by
	\begin{equation}\label{eq:ideal}
	\intI(\lambda)=\sum_{i\in I} e_i^{(\langle h_i,\lambda\rangle+1)}\intU_q(\mathfrak{n}).
	\end{equation}
	Set
	\[
	\intI'(\lambda)=\phi\big(\intI(\lambda)\big)=\sum_{i\in I}\intU_q(\mathfrak{n}_-)f_i^{(\langle h_i,\lambda\rangle +1)}
	\]
	to be the left ideal of $\intU_q(\mathfrak{n}_-)$. Then we have a $\mathbb{Z}[q^{\pm1}]$-module isomorphism
	\[
	\intU_q(\mathfrak{n}_-)/\intI'(\lambda)\xrightarrow{\sim}\intV(\lambda), \qquad x+\intI'(\lambda)\mapsto xv_\lambda.
	\]

	For $w\in W$ with a fixed reduced expression $\underline{w}=(i_1,\dots,i_r)$, define the extremal vector  $$v_{w\lambda}=f_{i_1}^{(a_1)}\dots f_{i_r}^{(a_r)}v_\lambda\in V(\lambda),$$ where $a_t=\langle h_{i_t},s_{i_{t+1}}\dots s_{i_r}(\lambda)\rangle$, for $1\leq t\leq r$. By the quantum Verma identity (\cite{Lu10}*{Proposition 39.3.7}), the element $v_{w\lambda}$ only depends on the weight $w\lambda$, and not on the choice of $w$ and its reduced expressions.
	
	Following \cite{Lu10}*{Proposition 19.1.2}, there is a unique non-degenerate symmetric bilinear form $(\cdot,\cdot)_\lambda: V(\lambda)\times V(\lambda)\rightarrow\mathbb{Q}(q)$ such that $$(v_\lambda,v_\lambda)_\lambda=1\qquad\text{and}\qquad (xu,v)_\lambda=(u,\phi(x)v)_\lambda,$$ for $x\in U_q(\mathfrak{g})$ and $u,v\in V(\lambda)$. Moreover, the bilinear form restricts to the integral form: $(\cdot,\cdot)_\lambda:\intV(\lambda)\times\intV(\lambda)\rightarrow\mathbb{Z}[q^{\pm1}]$ (\cite{Lu10}*{19.3.3 (c)}).
	
	By definition, for any $\lambda\in P_+$, $w\in W$ and $x\in\intI(\lambda)$, we have
	\begin{equation}\label{eq:pair}
	(xv_{w\lambda},v_\lambda)_\lambda=(v_{w\lambda},\phi(x)v_\lambda)_\lambda=0.
	\end{equation}
	
	For $\lambda\in P_+$ and $w,u\in W$, define the \emph{unipotent quantum minor}
	$D(w\lambda,u\lambda)$ in $A_q(\mathfrak{n})$ by $$D(w\lambda,u\lambda)(x)=(xv_{w\lambda},v_{u\lambda})_\lambda,\qquad \text{for }x\in U_q(\mathfrak{n}).$$
	
	It is proved in \cite{GLS}*{Proposition 6.3} that any (nonzero) unipotent quantum minor is a dual canonical basis element which can be described explicitly using the theory of crystal basis. The following lemma is a special case of \emph{loc. cit.} 
	
	\begin{lemma}\label{le:prca}
		For any $\lambda\in P_+$ and $w\in W$, there is a unique element $b$ in $\mathbf{B}_{\lambda-w\lambda}$ such that $b\not\in \intI(\lambda)$. One has
		\[
		b\in e_{i_1}^{(a_1)}\dots e_{i_r}^{(a_r)}+\intI(\lambda),
		\]
		where $\underline{w}=(i_r,\dots,i_1)$ is any reduced expression of $w$, and $a_t=\langle h_{i_t},s_{i_{t-1}}\dots s_{i_1}(\lambda)\rangle$, for $1\leq t\leq r$. Moreover we have $D(w\lambda,\lambda)(b')=\delta_{b,b'}$ for any $b'\in\mathbf{B}$.
	\end{lemma}
	
	\subsection{Quantum cluster algebra structures}\label{sec:qcs}
	
	We describe the quantum cluster algebra structure on the quantized coordinate rings in this subsection. To do that we firstly need to extend coefficients to $\mathbb{Z}[q^{\pm1/2}]$. Here $q^{1/2}$ is a square root of $q$ in the algebraic closure of $\mathbb{Q}(q)$. Set $$\intA_{q^{1/2}}(\mathfrak{n})=\mathbb{Z}[q^{\pm1/2}]\otimes_{\mathbb{Z}[q^{\pm1}]}\intA_q(\mathfrak{n})\quad \text{and}\quad \intU_{q^{1/2}}(\mathfrak{n})=\mathbb{Z}[q^{\pm1/2}]\otimes_{\mathbb{Z}[q^{\pm1}]}\intU_q(\mathfrak{n}).$$ We always view $\intA_q(\mathfrak{n})$ (resp., $\intU_q(\mathfrak{n})$) as a subset of $\intA_{q^{1/2}}(\mathfrak{n})$ (resp., $\intU_{q^{1/2}}(\mathfrak{n})$).
	
	Take any $w\in W$. Set $$\intA_{q^{1/2}}(\mathfrak{n}(w))=\mathbb{Z}[q^{\pm1/2}]\otimes_{\mathbb{Z}[q^{\pm1}]}\intA_q(\mathfrak{n}(w))\subset \intA_{q^{1/2}}(\mathfrak{n}).$$
	Since $\intA_{q^{1/2}}(\mathfrak{n}(w))$ is an Ore domain, one can embed $\intA_{q^{1/2}}(\mathfrak{n}(w))$ into its skew field of fractions, which we denote by $F_{q}(\mathfrak{n}(w))$ (\cite{GLS}*{\S 7.5}).
	
	Fix a reduced expression $\underline{w}=(i_r,\dots,i_1)$ of $w$. For $1\leq t\leq r$, define 
	\begin{equation}\label{eq:defl}
	\lambda_t=s_{i_1}\dots s_{i_t}(\varpi_{i_t})\qquad\text{and}\qquad  D_t=D(\lambda_t,\varpi_{i_t}).
	\end{equation}
	Then $D_t$ ($1\leq t\leq r$) belongs to $\intA_{q^{1/2}}(\mathfrak{n}(w))$ (\cite{GLS}*{Corollary 12.4} \cite{GY}*{Theorem 7.3}).
	
	Let $J=\{1,\dots,r\}$. Set 
	\[
	J_\fz=\{k\in J\mid i_s\neq i_k,\forall s>k\}\qquad\text{and}\qquad J_\ex=J\backslash J_\fz.
	\]
	Note that for $k\in J_\fz$ one has $\lambda_t=w^{-1}(\varpi_{i_t})$.
	
	By \cite{BZ05}*{Theorem 10.1}, we can take an integer-valued skew-symmetric matrix $\Lambda=(\lambda_{tk})_{t,k\in J}$ such that 
	\[
	D_tD_k=q^{\lambda_{tk}}D_tD_k.
	\]
	
	%One can explicitly construct the exchange matrix $\tB_{\underline{w}}$ of size $J\times J_\ex$. We will not give the definition of $\tB_{\underline{w}}$ but refer 
	
	For $t\in J$, set
	\begin{align*}
	t_+&=\text{min}(\{k\mid t<k\leq r,i_k=i_t\}\cup\{r+1\}),\\
	t_-&=\text{max}(\{k\mid 1\leq k<t,i_t=i_k\}\cup\{0\}).
	\end{align*}
	
	Following \cite{GY}*{Proposition 7.2}, define the integer-valued $J\times J_\ex$-matrix $\tB_{\underline{w}}$, 
	\begin{equation*}
	(\tB_{\underline{w}})_{tk}=\left\{\begin{array}{ll}
	1, & \text{if $t=k_-$,} \\
	-1, & \text{if $t=k_+$,} \\
	a_{{i_t}i_k}, & \text{if $t<k<t_+<k_+$,} \\
	-a_{i_ti_k}, & \text{if $k<t<k_+<t_+$,} \\
	0, & \text{otherwise.}
	\end{array}\right.
	\end{equation*}
	
	By \emph{loc.cot.}, we have:
	
	\emph{(a) The pair $(\Lambda,\tB_{\underline{w}})$ is a compatible pair with the skew-symmetrizer $D$, where the $J_\ex\times J_\ex$ part of $D$ is the diagonal matrix with diagonal entries $d_k=2\dt_{i_k}$ ($k\in J_\ex$).}
	
	\begin{theorem}{(\cites{KKKO17,GY})}
		The triple $\cL_{\underline{w}}=(\mathbf{x}=(q^{\gamma_t}D_t)_{t\in J},\Lambda,\tB_{\underline{w}})$ is a quantum seed in $F_q(\mathfrak{n}(w))$, so it defines a quantum cluster algebra $\cA_q(\cL_{\underline{w}})$. One has $$\cA_q(\cL_{\underline{w}})=\intA_{q^{1/2}}(\mathfrak{n}(w)).$$
		Here $\gamma_t\in\mathbb{Z}/2$ ($t\in J$). 
	\end{theorem}
	
	\subsection{Change of rings}\label{sec:bach}
	
	Let $R$ be any (unital) commutative $\mathbb{Z}[q^{\pm1/2}]$-algebra. Set $$\intU_R(\mathfrak{n})=R\otimes_{\mathbb{Z}[q^{\pm1/2}]}\intU_{q^{1/2}}(\mathfrak{n})\qquad \text{and}\qquad \intA_R(\mathfrak{n})=R\otimes_{\mathbb{Z}[q^{\pm1/2}]}\intA_{q^{1/2}}(\mathfrak{n}).$$  
	
	It is clear that $\intU_R(\mathfrak{n})$ has a $Q_+$-grading and $\intA_R(\mathfrak{n})$ can be identified with the graded dual, that is,
	\[
	\intA_R(\mathfrak{n})\cong \bigoplus_{\gamma\in Q_+}\text{Hom}_R(\intU_R(\mathfrak{n})_\gamma,R).
	\]
	
	Take $w\in W$ and define 
	$$ \intA_R(\mathfrak{n}(w))=R\otimes_{\mathbb{Z}[q^{\pm1/2}]}\intA_{q^{1/2}}(\mathfrak{n}(w)).$$
	
	Since $\intA_{q}(\mathfrak{n}(w))$ is spanned by a subset of dual canonical basis (\cite{Ki}*{Theorem 4.25}), the quotient $\intA_{q^{1/2}}(\mathfrak{n})/\intA_{q^{1/2}}(\mathfrak{n}(w))$ is a free $\mathbb{Z}[q^{\pm1/2}]$-module. Hence we can identify $\intA_R(\mathfrak{n}(w))$ as an $R$-subalgebra of $\intA_R(\mathfrak{n})$. 
	
	\subsection{Quantum Frobenius morphisms and their splittings}\label{sec:qfs}
	
	We recall the construction of quantum Frobenius morphisms and their splittings in this subsection.
	
	Let $l$ be a positive odd integer which is coprime to $\dt_i$, for $i\in I$. Let $\varepsilon$ be a primitive $l$-th root of unity. Define $\mathbb{Z}[q^{\pm1/2}]$-algebra $R_1$, $R_\varepsilon$ as following. We set $R_1=R_\varepsilon=\mathbb{Z}[\varepsilon]$ as rings, and the $\mathbb{Z}[q^{\pm1/2}]$-algebra structure on $R_1$ (resp., $R_\varepsilon$) is given by $q^{1/2}\mapsto 1$ (resp., $q^{1/2}\mapsto \varepsilon^{(l+1)/2}$). 
	
	Recall \S \ref{sec:bach}. We write $\intU_1(\mathfrak{n})=\intU_{R_1}(\mathfrak{n})$, $\intU_\varepsilon(\mathfrak{n})=\intU_{R_\varepsilon  }(\mathfrak{n})$, $\intA_1(\mathfrak{n})=\intA_{R_1}(\mathfrak{n})$ and $\intA_\varepsilon(\mathfrak{n})=\intA_{R_\varepsilon}(\mathfrak{n})$ for simplicity.
	
	For any element $x\in \intA_{q^{1/2}}(\mathfrak{n})$, write $$_\varepsilon x=1\otimes_{\mathbb{Z}[q^{\pm1/2}]}x\in\intA_\varepsilon(\mathfrak{n})\qquad\text{and}\qquad _1x=1\otimes_{\mathbb{Z}[q^{\pm1/2}]} x\in\intA_1(\mathfrak{n}).$$
	
	Similarly, for any $y\in \intU_{q^{1/2}}(\mathfrak{n})$, write $$_\varepsilon y=1\otimes_{\mathbb{Z}[q^{\pm1/2}]}y\in\intU_\varepsilon(\mathfrak{n})\qquad\text{and}\qquad _1y=1\otimes_{\mathbb{Z}[q^{\pm1/2}]} y\in\intU_1(\mathfrak{n}).$$
	
	The first part of the following theorem is originally due to Lusztig \cite{Lu90 } \cite{Lu10}*{35.1.7} with certain additional assumptions on $l$, and McGerty \cite{Mc} gave another construction using Hall algebras where the additional assumptions were removed. The second part is due to Lusztig \cite{Lu10}*{35.1.8}.
	\begin{theorem}
		(a) There exists a unique $\mathbb{Z}[\varepsilon]$-algebra homomorphism
		\[
		\ofr:\intU_\varepsilon(\mathfrak{n})\longrightarrow\intU_1(\mathfrak{n})
		\]
		such that $\ofr({_\varepsilon e_i^{(n)}})$ equals $_1e_i^{(n/l)}$ if $l$ divides $n$, and equals 0 if otherwise, for $i\in I$ and $n\in\mathbb{N}$.
		
		(b) There exists a unique $\mathbb{Z}[\varepsilon]$-algebra homomorphism
		\[
		\fr:\intU_1(\mathfrak{n})\longrightarrow\intU_\varepsilon(\mathfrak{n})
		\]
		such that $\fr({_1e_i^{(n)}})={_\varepsilon e_i^{(nl)}}$, for $i\in I$ and $n\in\mathbb{N}$.
	\end{theorem}
	
	Now we pass to the dual side. Define
	\[
	\dofr:\intA_1(\mathfrak{n})\longrightarrow \intA_\varepsilon(\mathfrak{n})\qquad\text{and}\qquad\dfr:\intA_\varepsilon(\mathfrak{n})\longrightarrow\intA_1(\mathfrak{n}),
	\]
	by $$\dofr(f)(x)=f(\ofr(x)) \qquad \text{for }  f \in\intA_1(\mathfrak{n}), x\in\intU_\varepsilon(\mathfrak{n}),$$ and $$\dfr(g)(y)=g(\fr(y)) \qquad \text{for } g\in\intA_\varepsilon(\mathfrak{n}),y\in \intU_1(\mathfrak{n}).$$ These two maps are well-defined $\mathbb{Z}[\varepsilon]$-linear maps by weight consideration.
	
	It follows immediately from definitions that these two maps are compatible with $Q_+$-grading, that is,
	\begin{equation}\label{eq:homo}
	\dofr\big(\intA_1(\mathfrak{n})_\mu\big)\subset \intA_\varepsilon(\mathfrak{n})_{l\mu}\qquad\text{and}\qquad \dfr\big(\intA_\varepsilon(\mathfrak{n})_\mu\big)\subset \intA_1(\mathfrak{n})_{\mu/l},
	\end{equation}
	for any $\mu\in Q_+$. Here $\intA_1(\mathfrak{n})_{\mu/l}$ is understood as zero space if $\mu\not\in lQ_+$.
	
	\begin{prop}\label{prop:Frli}
		The map $\dofr$ is a $\mathbb{Z}[\varepsilon]$-algebra homomorphism. The map $\dfr$ satisfies $$\dfr\circ\dofr=id\qquad \text{and}\qquad \dfr\big(\dofr(f)g\big)=f\dfr(g),$$ for $f\in \intA_1(\mathfrak{n})$ and $g\in\intA_\varepsilon(\mathfrak{n})$.
	\end{prop}
	
	\begin{proof}
		Let $r_\varepsilon:\intU_\varepsilon(\mathfrak{n})\rightarrow \intU_\varepsilon(\mathfrak{n})\otimes \intU_\varepsilon(\mathfrak{n})$ and  $r_1:\intU_1(\mathfrak{n})\rightarrow \intU_1(\mathfrak{n})\otimes \intU_1(\mathfrak{n})$ be the base changes of the coproduct $r$, where the tensor products are over $\mathbb{Z}[\varepsilon]$. The twisted algebra structure naturally extends to $\intU_\varepsilon(\mathfrak{n})\otimes\intU_\varepsilon(\mathfrak{n})$ and $\intU_1(\mathfrak{n})\otimes\intU_1(\mathfrak{n})$. Then $r_\varepsilon$ and $r_1$ are both algebra homomorphisms.
		
		By definition, the map $\dofr$ is an algebra homomorphism if and only if we have 
		\begin{equation}\label{eq:frt}
		(\ofr\otimes\ofr)\circ r_\varepsilon=r_1\circ \ofr
		\end{equation}
		as maps from $\intU_\varepsilon(\mathfrak{n})$ to $\intU_1(\mathfrak{n})\otimes\intU_1(\mathfrak{n})$. Thanks to \eqref{eq:cop} it is direct to check \eqref{eq:frt} when acting on divided powers. Since maps $\ofr$, $r_1$ and $r_\varepsilon$ are all algebra homomorphisms, it will suffice to show $\ofr\otimes\ofr$ is also an algebra homomorphism (with respect to the twisted product structures on tensor products). 
		
		Take $x_1$, $x_2$, $y_1$ and $y_2$ to be any homogeneous elements in $\intU_\varepsilon(\mathfrak{n})$. Then
		\begin{align*}
		(\ofr\otimes\ofr)\big((x_1\otimes x_2)\cdot (y_1\otimes y_2)\big)&=\varepsilon^{-(\wt(x_2),\wt(y_1))}\ofr(x_1y_1)\otimes\ofr(x_2y_2)\\
		&\stackrel{(\heartsuit )}{=}\ofr(x_1y_1)\otimes\ofr(x_2y_2)\\&=(\ofr\otimes\ofr)(x_1\otimes x_2)\cdot(\ofr\otimes\ofr)(y_1\otimes y_2).
		\end{align*}
		The equality $(\heartsuit)$ follows by noticing that both sides are zero unless $\wt(x_2)$ and $\wt(y_1)$ both belong to $lQ_+$. Therefore equality \eqref{eq:frt} holds and $\dofr$ is an algebra homomorphism.
		
		We prove the second statement. Since $\ofr\circ\fr=id$, we have $\dfr\circ\dofr=id$. By definition, the last equality holds if and only if the following diagram commutes \begin{equation}\label{dia:prfr}
		\begin{tikzcd}
		& \intU_1(\mathfrak{n}) \arrow[r,"\fr"] \arrow[d,"r_1"'] & \intU_\varepsilon(\mathfrak{n})\arrow[r,"r_\varepsilon"] &\intU_\varepsilon(\mathfrak{n})\otimes\intU_\varepsilon(\mathfrak{n}) \arrow[d,"\ofr\otimes id"]\\& \intU_1(\mathfrak{n})\otimes\intU_1(\mathfrak{n})\arrow[rr,"id\otimes\fr"'] & & \intU_1(\mathfrak{n})\otimes\intU_\varepsilon(\mathfrak{n}).
		\end{tikzcd}
		\end{equation}
		Here tensor products are over $\mathbb{Z}[\varepsilon]$. We endow $\intU_1(\mathfrak{n})\otimes\intU_\varepsilon(\mathfrak{n})$ with the natural algebra structure (without twisting), and claim that all the maps in the diagram \eqref{dia:prfr} are algebra homomorphisms. We only prove this claim for the map $\ofr\otimes id$. The proof for other maps is trivial.
		
		Take homogeneous elements $x_i$, $y_i$ $(i=1,2)$ in $\intU_\varepsilon(\mathfrak{n)}$. Then
		\begin{align*}
		(\ofr\otimes id)\big((x_1\otimes x_2)\cdot(y_1\otimes y_2)\big)&=(\ofr\otimes id)(\varepsilon^{-(\wt(x_2),\wt(y_1))}x_1y_1\otimes x_2y_2)\\&=\ofr(x_1y_1)\otimes(x_2y_2)\\&=(\ofr\otimes id)(x_1\otimes x_2)\cdot(\ofr\otimes id)(y_1\otimes y_2),
		\end{align*}
		where the second equality follows by noticing both sides are zero unless $\wt(y_1)\in lQ_+$.
		
		It then remains to check the diagram \eqref{dia:prfr} when acting on divided powers, which is direct and will be skipped. 
	\end{proof}
	
	It follows from the above proposition that we can endow $\intA_\varepsilon(\mathfrak{n})$ with the $\intA_1(\mathfrak{n})$-module structure via the map $\dofr$. Then $\dfr$ is a splitting of the map $\dofr$ as $\intA_1(\mathfrak{n})$-module homomorphisms.
	
	\subsection{Proof of the main theorem}\label{sec:proof}
	
	We prove Theorem \ref{thm:main} in this subsection. Retain the same notation as in \S \ref{sec:bach} and \S \ref{sec:qfs}. 
	
	\begin{lemma}\label{le:qeo}
		Suppose $R$ is an integral domain and $w$ is a Weyl group element. Then we have:
		
		(a) the $R$-algebra $\intA_R(\mathfrak{n}(w))$ is an integral domain;
		
		(b) for any $f\in\intA_R(\mathfrak{n})$, if $g\cdot f\in\intA_R(\mathfrak{n}(w))$ for some nonzero element $g$ in $\intA_R(\mathfrak{n}(w))$, then $f\in\intA_R(\mathfrak{n}(w))$;
		
		(c) for any $f\in\intA_R(\mathfrak{n})$, if $g\cdot f=0$ for some nonzero element $g$ in $\intA_R(\mathfrak{n}(w))$, then $f=0$.
	\end{lemma}
	
	\begin{proof}
		Since $\intA_R(\mathfrak{n}(w))$ is a free $R$-module, it will suffice to prove the lemma over its field of fractions. We may assume $R$ is a field. 
		
		Take a reduced expression $\underline{w}=(i_r,\dots,i_1)$ and use the same notations as in \S \ref{sec:PBW}. For $\mathbf{n}, \mathbf{m}\in\mathbb{Z}^r$, it follows from the dual Levendorskii--Soibelman formula \cite{Ki}*{Theorem 4.27} that
		\begin{equation}\label{eq:prog}
		e_{\underline{w}}^*(\mathbf{n})e_{\underline{w}}^*(\mathbf{m})\in q^{\mathbb{Z}}e_{\underline{w}}^*(\mathbf{n}+\mathbf{m})+\sum_{\mathbf{n}'<\mathbf{n}+\mathbf{m}}\mathbb{Z}[q^{\pm1}]e_{\underline{w}}^*(\mathbf{n}')\qquad \text{in $\intA_q(\mathfrak{n}(w))$.}
		\end{equation}
		Here $<$ is the lexicographic order on $\mathbb{Z}^r$.
		
		Take $f$ and $g$ to be two nonzero elements in $\intA_R(\mathfrak{n}(w))$. By rescaling, we may assume that 
		\[
		f\in 1\otimes e_{\underline{w}}^*(\mathbf{n}_1)+\sum_{\mathbf{n}'<\mathbf{n}_1}R(1\otimes e_{\underline{w}}^*(\mathbf{n}')),\quad g\in 1\otimes e_{\underline{w}}^*(\mathbf{n}_2)+\sum_{\mathbf{n}'<\mathbf{n}_2}R(1\otimes e_{\underline{w}}^*(\mathbf{n}')).
		\]
		Then by \eqref{eq:prog}, we have 
		\[
		fg\in 1\otimes e_{\underline{w}}^*(\mathbf{n}_1+\mathbf{n}_2)+\sum_{\mathbf{n}'<\mathbf{n}_1+\mathbf{n}_2}R(1\otimes e_{\underline{w}}^*(\mathbf{n}')).
		\]
		In particular, we deduce that $fg$ is nonzero, so $\intA_R(\mathfrak{n}(w))$ is an integral domain. This proved the assertion (a).
		
		The assertion (b) and (c) follow from (a) and the tensor product decomposition \eqref{eq:tens}.
	\end{proof}
	
	\begin{lemma}\label{prop:vani}
		For any $f,g\in\intA_\varepsilon(\mathfrak{n})$, we have $\dfr(fg)=\dfr(gf)$.
	\end{lemma}
	
	\begin{proof}   
		Let $r_\varepsilon=\intU_\varepsilon(\mathfrak{n})\rightarrow \intU_\varepsilon(\mathfrak{n})\otimes \intU_\varepsilon(\mathfrak{n})$ be the base change of the coproduct $r$ as before, and $s:\intU_\varepsilon(\mathfrak{n})\otimes \intU_\varepsilon(\mathfrak{n})\rightarrow\intU_\varepsilon(\mathfrak{n})\otimes \intU_\varepsilon(\mathfrak{n})$ be the linear map given by $s(x\otimes y)=y\otimes x$. Here the tensor products are all over $\mathbb{Z}[\varepsilon]$.
		
		By definition it will suffice to prove that 
		\begin{equation}\label{eq:rs}
		r_\varepsilon\circ\fr(x)=s\circ r_\varepsilon\circ \fr(x),
		\end{equation}
		for any $x\in \intU_1(\mathfrak{n})$. 
		
		Let $M$ be the $\mathbb{Z}[\varepsilon]$-subspace of $\intU_\varepsilon(\mathfrak{n})\otimes\intU_\varepsilon(\mathfrak{n})$ spanned by elements of the form $x\otimes y$, where $x$ and $y$ are homogeneous elements in $\intU_\varepsilon(\mathfrak{n})$ such that the summation $\wt(x)+\wt(y)$ belongs to $lQ_+$. It is clear that the subspace $M$ is moreover a subalgebra and $s(M)=M$.
		
		Notice that for any homogeneous elements $x_1$, $x_2$, $y_1$ and $y_2$ in $\intU_\varepsilon(\mathfrak{n})$ with $\wt(x_1)+\wt(x_2)\in lQ_+$ and $\wt(y_1)+\wt(y_2)\in lQ_+$, we have
		\begin{align*}
		&s((x_1\otimes x_2)\cdot (y_1\otimes 
		y_2))=\varepsilon^{-(\wt(x_2),\wt(y_1))}x_2y_2\otimes x_1y_1\\&=\varepsilon^{-(\wt(x_1),\wt(y_2))}(x_2y_2\otimes x_1y_1)=s(x_1\otimes x_2)\cdot s(y_1\otimes y_2).
		\end{align*}
		Therefore $s\mid_M:M\rightarrow M$ is an algebra homomorphism.
		
		It is direct to see that $r_\varepsilon\circ\fr(\intU_1(\mathfrak{n}))\subset M$. Hence both $s\circ r_\varepsilon\circ\fr$ and $r_\varepsilon\circ\fr$ are algebra homomorphisms. It will suffice to check \eqref{eq:rs} for $x$ equaling divided powers, which follows immediately from \eqref{eq:cop}.
	\end{proof}
	
	For any $w\in W$, set $\intA_{\varepsilon}(\mathfrak{n}(w))=\intA_{R_\varepsilon}(\mathfrak{n}(w))$ and $\intA_1(\mathfrak{n}(w))=\intA_{R_1}(\mathfrak{n}(w))$. Fix a reduced expression $\underline{w}=(i_r,\dots,i_1)$ of $w$. Let $J=\{1,\dots,r\}$. Recall the partition $J=J_\ex\sqcup J_\fz$ in \S \ref{sec:qcs} and the unipotent quantum minors $D_t\in\intA_q(\mathfrak{n}(w))$ ($t\in J$) in \S \ref{sec:minor}. 
	
	Recall the initial quantum seed $\cL_{\underline{w}}=(\mathbf{x}=(q^{\gamma_t}D_t)_{t\in J},\Lambda,\tB_{\underline{w}})$ of $\intA_{q^{1/2}}(\mathfrak{n}(w))$ in \S \ref{sec:qcs}. In particular, the tuple $\mathbf{D}=(D_t)_{t\in J}$ is $\Lambda$-commutative. Define elements $\mathbf{D}^{\mathbf{a}}$ ($\mathbf{a}\in \mathbb{Z}^J$) in $F_q(\mathfrak{n}(w))$ as in \eqref{eq:twpro}. For any $\mathbf{a}\in\mathbb{Z}^J$, write $\mathbf{a}/l\in \mathbb{Z}^J$ to be the tuple which equals $\mathbf{0}$ if $\mathbf{a}\not\in l\mathbb{Z}^J$ and equals $(a_t/l)_{t\in J}$ if $\mathbf{a}\in l\mathbb{Z}^J$.
	
	\begin{prop}\label{prop:le1}
		For any $\mathbf{a}\in\mathbb{N}^J$, we have $$\dofr\big({_1\mathbf{D}^\mathbf{a}}\big)={_\varepsilon \mathbf{D}^{l\mathbf{a}}}\qquad\text{and}\qquad \dfr\big({_\varepsilon \mathbf{D}^{\mathbf{a}}}\big)={_1\mathbf{D}^{\mathbf{a}/l}}.$$
	\end{prop}
	
	\begin{proof}
		The case when $w=e$ is trivial. We assume $w\neq e$.
		
		For any $\lambda\in P_+$, write 
		$$\intV_\varepsilon(\lambda)=R_\varepsilon\otimes_{\mathbb{Z}[q^{\pm1}]}\intV(\lambda)\qquad\text{and}\qquad\intV_1(\lambda)=R_1\otimes_{\mathbb{Z}[q^{\pm1}]}\intV(\lambda).$$
		Let $_\varepsilon \big(\cdot,\cdot\big)_\lambda$ (resp., $_1\big(\cdot,\cdot\big)_\lambda$) be the $\mathbb{Z}[\varepsilon]$-bilinear form on $\intV_\varepsilon(\lambda)$ (resp., $\intV_1(\lambda)$) obtained from base change of $\big(\cdot,\cdot\big)_\lambda:\intV(\lambda)\times\intV(\lambda)\rightarrow\mathbb{Z}[q^{\pm1}]$.
		
		We then prove the first equality. Notice that by \eqref{eq:qprod}, we have
		\[
		\ D_t^l D_{t'}^l=q^{l^2\Lambda(\be_t,\be_{t'})}D_{t'}^lD_t^l \qquad \text{in $\intA_{q^{1/2}}(\mathfrak{n})$},
		\]
		for $ t,t'\in J$. Here $\be_k$ ($k\in J$) is the standard basis element in $\mathbb{Z}^J$. Therefore after specialisation, element ${_\varepsilon D_t^l}$ (${t\in J 
		}$) commutes with each other in $\intA_\varepsilon(\mathfrak{n})$. Since $\dofr$ is an algebra homomorphism, it will suffice to show for any $t\in J$,
		\begin{equation}\label{eq:cfr}
		\dofr({_1D_t})={_\varepsilon D_t^l}.
		\end{equation}
		
		By \cite{KKKO17}*{Corollary 9.1.3}, one has
		\[
		D_t^l=D(\lambda_t,\varpi_{i_t})^l=q^{-l(l-1)/2\cdot(\varpi_{i_t},\varpi_{i_t}-\lambda_t)}D(l\lambda_t,l\varpi_{i_t})\qquad \text{in $\intA_{q^{1/2}}(\mathfrak{n})$}.
		\]
		Here $\lambda_t=s_{i_1}\dots s_{i_t}\varpi_{i_t}=w^{-1}\varpi_{i_t}$ as in \eqref{eq:defl}.
		
		Hence in $\intA_\varepsilon(\mathfrak{n})$, 
		\[
		_\varepsilon D_t^l={_\varepsilon D(l\lambda_t,l\varpi_{i_t})}.
		\]
		
		By Lemma \ref{le:prca}, take $b\in \mathbf{B}$ such that $ D(l\lambda_t,l\varpi_{i_t})(b')=\delta_{b,b'}$, for $b'\in\mathbf{B}$. Then after base change we have
		$$ _\varepsilon D_t^l({_\varepsilon b'})=_\varepsilon D(l\lambda_t,l\varpi_{i_t})({_\varepsilon} b')=\delta_{b,b'},$$
		for $b'\in \mathbf{B}$.
		
		%the unique canonical basis element of weight ${\l\varpi_{i_t}}-l\lambda_t$ which does not belong to $\sum_{i\in I}e_i^{(\langle l\varpi_{i_s},\alpha_i\rangle+1)}U_q(\mathfrak{n})$. 
		
		One the other hand, the element $\dofr({_1D_t})\in \intA_\varepsilon(\mathfrak{n})$ is given by
		\[
		\dofr({_1D_t})(f)=(_1 D_t)(\ofr(f))={_1\big(\ofr(f)v_{\lambda_t},v_{\varpi_{i_t}}\big)_{\varpi_{i_t}}},\qquad \text{for $f\in \intU_\varepsilon(\mathfrak{n})$.}
		\]
		Here we use the same notations to denote images of extremal vectors after base changes. For any $b'\in \mathbf{B}$, it follows from weight considerations that $\dofr({_1D_t})({_1 b'})=0$ unless $\wt(b')=l\varpi_{i_t}-l\lambda_t$. 
		
		Recall the left ideal $\intI(\lambda)$ ($\lambda\in P_+$) from \eqref{eq:ideal}. By definition, one has
		\begin{equation}\label{eq:frR}
		\ofr\big(R_\varepsilon\otimes \intI(l\varpi_{i_t})\big)\subset R_1\otimes\intI(\varpi_{i_t}).
		\end{equation}
		
		Take $b'\in \mathbf{B}_{l\varpi_{i_t}-l\lambda_t}$. If $b'\neq b$, by Lemma \ref{le:prca} we have $b'\in\intI(l\varpi_{i_t})$. Combining \eqref{eq:frR} and \eqref{eq:pair}, we have
		\[
		\dofr({_1D_t})({_\varepsilon b'})={_1\big(\ofr({_\varepsilon}b')}v_{\lambda_t},v_{\varpi_{i_t}}\big)_{\varpi_{i_t}}=0.
		\]
		
		By Lemma \ref{le:prca}, we have
		\[
		b\in e_{i_t}^{(la_t)}\dots e_{i_1}^{(la_1)}+\intI(l\varpi_{i_t}),
		\]
		where $a_k=\langle h_{i_k},s_{i_{k-1}}\dots s_{i_t}(\varpi_{i_t})\rangle$, for $1\leq k\leq t$. Hence by \eqref{eq:frR} and \eqref{eq:pair}, we have
		\[
		\dofr({_1D_t})({_\varepsilon}b)={_1\big(\ofr({_\varepsilon}b)}v_{\lambda_t},v_{\varpi_{i_t}}\big)_{\varpi_{i_t}}={_1\big({_1}(e_{i_t}^{(a_t)}\dots e_{i_1}^{(a_1)})v_{\lambda_t}},v_{\varpi_{i_t}}\big)_{\varpi_{i_t}}=1.
		\]
		In conclusion we have $\dofr({_1D_t})({_\varepsilon}b')={_\varepsilon D_t^l}({_\varepsilon b'})$ for $b'\in \mathbf{B}$. Therefore \eqref{eq:cfr} is proved.
		
		We next prove the second equality. Firstly suppose $\mathbf{a}=l\mathbf{a}'$ for some $\mathbf{a}'\in\mathbb{N}^r$. Then thanks to Proposition \ref{prop:Frli} and the previous discussion, we have
		\[
		\dfr({_\varepsilon \mathbf{D}^{\mathbf{a}}})=\dfr\circ \dofr({_1 \mathbf{D}^{\mathbf{a}'}})={_1 \mathbf{D}^{\mathbf{a}'}}.
		\]
		
		Now suppose $\mathbf{a}\not\in l\mathbb{N}^r$. We need to show that $\dfr({_\varepsilon \mathbf{D}^\mathbf{a}})=0$. 
		
		Recall $J_\fz=\{t\in [1,r]\mid i_k\neq i_t,\forall k>t\}$ and $J_\ex=[1,r]\backslash J_\fz$. We divide into two cases depending on divisibility of $a_t$ ($t\in J_\ex$) by $l$.
		
		\underline{Case I}: $l\mid a_t$, for all $t\in J_\ex$. 
		
		Write $\mathbf{a}=l\mathbf{a}'+\mathbf{a}''$, with $\mathbf{a}'=(a_t')_{t\in J},\mathbf{a}''=(a''_t)_{t\in J}$ in $\mathbb{N}^J$ and $0\leq a''_t<l$ ($1\leq t\leq r$). Then $a_t''=0$ for $t\in J_\ex$ and $\mathbf{a}''\neq \mathbf{0}$ by our assumption. By \eqref{eq:qprod} we have
		\[
		\mathbf{D}^{l\mathbf{a}'}\mathbf{D}^{\mathbf{a}''}=q^{1/2\Lambda(l\mathbf{a}',\mathbf{a}'')}\mathbf{D}^{\mathbf{a}}\qquad\text{in $\intA_{q^{1/2}}(\mathfrak{n})$}.
		\]
		Hence after base change we have
		$
		{_\varepsilon \mathbf{D}^\mathbf{a}}={_\varepsilon \mathbf{D}^{l\mathbf{a}'}}{_\varepsilon \mathbf{D}^{\mathbf{a}''}}$ in $\intA_\varepsilon(\mathfrak{n})$.
		
		It then follows from Proposition \ref{prop:Frli} that
		\[
		\dfr({_\varepsilon \mathbf{D}^{\mathbf{a}}})=\dfr({_\varepsilon \mathbf{D}^{l\mathbf{a}'}}{_\varepsilon \mathbf{D}^{\mathbf{a}''}})=\dfr(\dofr({_1 \mathbf{D}^{\mathbf{a}'}}){_\varepsilon \mathbf{D}^{\mathbf{a}''}})={_1 \mathbf{D}^{\mathbf{a}'}}\cdot\dfr({_\varepsilon \mathbf{D}^{\mathbf{a}''}}).
		\]
		Hence it will suffice to show $\dfr({_\varepsilon \mathbf{D}^{\mathbf{a}''}})=0$.
		
		Recall $\lambda_t=s_{i_1}\dots s_{i_t}\varpi_{i_t}=w^{-1}\varpi_{i_t}$ for $t\in J_\fz$. Since $\mathbf{a}''$ only contains the $\be_t$ factors for those $t\in J_\fz$, the monomial $\mathbf{D}^{\mathbf{a}''}$ is a product of the quantum minors of the form $D(a''_t(w^{-1}\varpi_{i_t}),a_t''\varpi_{i_t})$ for various $t\in J_\fz$. It follows from \cite{KKKO17}*{Corollary 9.1.3} that 
		\[
		_\varepsilon \mathbf{D}^{\mathbf{a}''}=\varepsilon^N{_\varepsilon D(w^{-1}\mu,\mu)}\qquad \text{in $\intA_\varepsilon(\mathfrak{n})$}
		\]
		for some $N\in\mathbb{Z}$, and $\mu\in P_+$ where $\langle h_i,\mu\rangle<l$, for $i\in I$.
		
		Take any $f=e_{j_1}^{(b_1)}\dots e_{j_m}^{(b_m)}$ in $\intU_q(\mathfrak{n})$. We have
		\begin{align*}
		\dfr(_\varepsilon D(w^{-1}\mu,\mu))({_1}f)&={_1\big(\fr(f)v_{w^{-1}\mu},v_\mu\big)}_\mu\\&={_1\big(e_{j_1}^{(lb_1)}\dots e_{j_m}^{(lb_m)}v_{w^{-1}\mu},v_\mu\big)_\mu}\\&={_1\big(v_{w^{-1}\mu},f_{j_m}^{(lb_m)}\dots f_{j_1}^{(lb_1)}v_\mu\big)_\mu}.
		\end{align*}
		By our condition on $\mu$, the vector $f_{j_m}^{(lb_m)}\dots f_{j_1}^{(lb_1)}v_\mu$ equals $0$ unless $b_1=\dots=b_m=0$, in which case $f=1$. In any cases, the element $\big(e_{j_1}^{(lb_1)}\dots e_{j_m}^{(lb_m)}v_{w^{-1}\mu},v_\mu\big)_\mu$ equals $0$. Therefore we deduce that $\dfr({_\varepsilon \mathbf{D}^{\mathbf{a}''}})=0$.
		
		\underline{Case II}: $l\nmid a_t$, for some $t\in J_\ex$. We fix such an index $t$.
		
		Write $\tB_{\underline{w}}=(b_{ij})_{i\in J,j\in J_\ex}$. Let $\mathbf{b}_t=\sum_{i=1}^rb_{it}\be_i\in \mathbb{Z}^J$ be the $t$-th column vector. Since $(\Lambda,\tB_{\underline{w}})$ is compatible with skew-symmetrizer described in \S \ref{sec:qcs} (a), we have
		\[
		\Lambda(\mathbf{b}_t,\be_k)=2\dt_{i_k}\delta_{tk}\qquad \text{for $1\leq k\leq r$.}
		\]
		Take $\mathbf{b}_t'\in\mathbb{N}^J$ such that $\mathbf{b}_t'-\mathbf{b}_t\in l\mathbb{Z}^J$. Take $\mathbf{a}'\in\mathbb{N}^J$ such that $\mathbf{c}=\mathbf{a}+l\mathbf{a}'-\mathbf{b}_t'\in\mathbb{N}^J$. Then we have
		\[
		\mathbf{D}^{\mathbf{c}}\mathbf{D}^{\mathbf{b}_t'}-\mathbf{D}^{\mathbf{b}_t'}\mathbf{D}^{\mathbf{c}}=(q^{\Lambda(\mathbf{c},\mathbf{b}_t')}-q^{-\Lambda(\mathbf{c},\mathbf{b}_t')})\mathbf{D}^{\mathbf{a}+l\mathbf{a}'}\qquad \text{in $\intA_{q^{1/2}}(\mathfrak{n})$.}
		\]
		
		Hence by Lemma \ref{prop:vani} and Proposition \ref{prop:Frli} we have in $\intA_\varepsilon(\mathfrak{n})$
		\begin{align}
		0=\dfr({_\varepsilon}\mathbf{D}^{\mathbf{c}}{_\varepsilon}\mathbf{D}^{\mathbf{b}_t'}-{_\varepsilon}\mathbf{D}^{\mathbf{b}_t'}{_\varepsilon}\mathbf{D}^{\mathbf{c}})&=(\varepsilon^{\Lambda(\mathbf{c},\mathbf{b}_t')}-\varepsilon^{-\Lambda(\mathbf{c},\mathbf{b}_t')})\dfr({_\varepsilon \mathbf{D}}^{\mathbf{a}+l\mathbf{a}'})\notag\\&=(\varepsilon^{\Lambda(\mathbf{c},\mathbf{b}_t')}-\varepsilon^{-\Lambda(\mathbf{c},\mathbf{b}_t')})(_1 \mathbf{D}^{\mathbf{a}'})\dfr({_\varepsilon \mathbf{D}}^{\mathbf{a}})\label{eq:zerof}.
		\end{align}
		Notice that modulo $l$ we have
		\[
		\Lambda(\mathbf{c},\mathbf{b}_t')=\Lambda(\mathbf{a}-\mathbf{b}_t',\mathbf{b}_t')=\Lambda(\mathbf{a},\mathbf{b}_t')=\Lambda(\mathbf{a},\mathbf{b}_t)=-2\dt_{i_t}a_t,
		\]
		which is nonzero since $(l,2\dt_{i_t})=1$ and $l\nmid a_t$ by our assumption.
		
		Therefore $\varepsilon^{\Lambda(\mathbf{c},\mathbf{b}_t')}-\varepsilon^{-\Lambda(\mathbf{c},\mathbf{b}_t')}\neq 0$. Since $_1\mathbf{D}^{\mathbf{a}'}$ is a nonzero element in $\intA_1(\mathfrak{n}(w))$, by Lemma \ref{le:qeo} (c) and \eqref{eq:zerof} we deduce that $\dfr({_\varepsilon \mathbf{D}^{\mathbf{a}}})=0$. We finish the proof.
	\end{proof}
	
	\begin{corollary}\label{cor:rest}
		For any $w\in W$, we have 
		\[
		\dfr\big(\intA_\varepsilon(\mathfrak{n}(w))\big)\subset \intA_1(\mathfrak{n}(w))\qquad\text{and}\qquad  \dofr\big(\intA_1(\mathfrak{n}(w))\big)\subset \intA_\varepsilon(\mathfrak{n}(w)).
		\]
	\end{corollary}
	
	\begin{proof}
		
		Retain the same notations as before. Take any $f\in \intA_{q^{1/2}}(\mathfrak{n}(w))$. By the Laurent phenomenon, we have
		\[
		f\in \mathbb{Z}[q^{\pm1/2}]\langle D_t^{\pm1}\mid 1\leq t\leq r\rangle\qquad \text{in $F_{q}(\mathfrak{n}(w))$.}
		\]
		We take $\mathbf{a}\in\mathbb{N}^r$ with entries large enough such that 
		\[
		\mathbf{D}^{l\mathbf{a}}f\in\mathbb{Z}[q^{\pm1/2}]\langle D_t\mid 1\leq t\leq r\rangle\qquad\text{in $\intA_{q^{1/2}}(\mathfrak{n}(w)).$}
		\]
		Then in $\intA_\varepsilon(\mathfrak{n}(w))$, the product $_\varepsilon\mathbf{D}^{l\mathbf{a}}{_\varepsilon f}$ is a $\mathbb{Z}[\varepsilon]$-linear combination of elements in the form $_\varepsilon \mathbf{D}^{\mathbf{a}}$ ($\mathbf{a}\in\mathbb{N}^r$). By Proposition \ref{prop:Frli} and Proposition \ref{prop:le1}, we deduce that
		\[
		(_1\mathbf{D}^\mathbf{a})\dfr({_\varepsilon f})=\dfr(\dofr({_1 \mathbf{D}^{\mathbf{a}}}){_\varepsilon f)}= \dfr(_\varepsilon\mathbf{D}^{l\mathbf{a}}{_\varepsilon f})\in \intA_1(\mathfrak{n}(w)).
		\]
		Note that $_1\mathbf{D}^\mathbf{a}$ belongs to $\intA_1(\mathfrak{n}(w))$. By Lemma \ref{le:qeo} (b) we deduce that $\dfr({_\varepsilon f})$ also belongs to $\intA_1(\mathfrak{n}(w))$. Hence we have $\dfr\big(\intA_\varepsilon(\mathfrak{n}(w))\big)\subset \intA_1(\mathfrak{n}(w))$.
		
		One can show another inclusion in a similar way by using the fact that $\dofr$ is an algebra homomorphism.
	\end{proof}
	
	We are now ready to prove the main theorem.
	
	\begin{proof}[Proof of Theorem \ref{thm:main}]
		Let $\cL=\mu_{k_s}\dots\mu_{k_1}(\cL_{\underline{w}})$ be a quantum seed of $\intA_{q^{1/2}}(\mathfrak{n}(w))$ which is mutation-equivalent to the initial seed $\cL_{\underline{w}}$, where $(k_s,\dots,k_1)$ is a sequence of exchangeable indices. We prove the theorem by induction on $s$.
		
		\underline{Base case:} Suppose $s=0$. Then $\cL=\cL_{\underline{w}}=(\mathbf{x}=(x_t)_{t\in J},\Lambda,\tB_{\underline{w}})$, where $x_t=q^{\gamma_t}D_t$ with $\gamma_t\in\mathbb{Z}/2$. Take $\mathbf{a}=(a_t)_{t\in J}\in\mathbb{N}^J$. Thanks to Proposition \ref{prop:le1}, we have 
		\begin{align*}
		&\dofr({_1\mathbf{x}^{\mathbf{a}}})=\dofr({_1\mathbf{D}^{\mathbf{a}}})={_\varepsilon \mathbf{D}^{l\mathbf{a}}}={_\varepsilon \mathbf{x}^{l\mathbf{a}}}\\\text{and}\qquad & \dfr({_\varepsilon\mathbf{x}^{\mathbf{a}}})=\big(\prod_{t\in J}\varepsilon^{a_t\gamma_t}\big)\dfr({_\varepsilon\mathbf{D}^\mathbf{a}})\stackrel{(\heartsuit )}{=}{_1\mathbf{D}^{\mathbf{a}/l}}={_1\mathbf{x}^{\mathbf{a}/l}}.
		\end{align*}
		Here $(\heartsuit)$ follows by noticing that both sides are zero unless $\mathbf{a}\in lQ_+$, in which case $\varepsilon^{a_t\gamma_t}=1$ for $t\in J$. Hence the theorem is prove for $s=0$.
		
		\underline{Induction step:} Suppose the theorem holds for the quantum seed $\cL'=(\mathbf{x}'=(x_t')_{t\in J},\Lambda',\tB')$. We show it holds for the quantum seed $\cL''=\mu_k(\cL')=(\mathbf{x}''=(x_t'')_{t\in J},\Lambda'',\tB'')$ for any $k\in J_\ex$.
		
		Take $\mathbf{a}=(a_t)_{t\in J}\in \mathbb{N}^J$. We firstly  prove 
		$
		\dofr({_1(\mathbf{x}'')^\mathbf{a}})={_\varepsilon(\mathbf{x}'')^{l\mathbf{a}}}, \text{ for $\mathbf{a}\in\mathbb{N}^J$.}
		$ Since $\dofr$ is an algebra homomorphism, it suffices to show that
		\begin{equation}\label{eq:nufr}
		\dofr({_1x''_t})={_\varepsilon (x''_t)^l}\quad \text{for $t\in J$.}
		\end{equation}

		Suppose $t\neq k$. Then $x''_t=x'_t$, and \eqref{eq:nufr} follows immediately from our assumption.
		
		Now suppose $t=k$. Let $\mathbf{b}_k'=\sum_{i=1}^rb'_{ik}\be_i\in \mathbb{Z}^J$ be the $k$-th column vector. Here $b'_{ik}$ is the $(i,k)$-entry of $\tB'$ as usual. By definition, we have
		\[
		x_k''=(\mathbf{x}')^{-\be_k+[\mathbf{b}_k']_+}+(\mathbf{x}')^{-\be_k+[-\mathbf{b}_k']_+}\qquad \text{in $F_q(\mathfrak{n}(w))$.}
		\]
		
		By \S \ref{sec:qcs} (a) we have
		\[
		\Lambda'(-\be_k+[\mathbf{b}'_k]_+,-\be_k+[-\mathbf{b}'_k]_+)=\Lambda'(\mathbf{b}_k',-\be_k+[-\mathbf{b}_k']_+)=-2\dt_{i_k}.
		\]
		
		One has
		\[
		(\mathbf{x}')^{-\be_k+[\mathbf{b}_k']_+}(\mathbf{x}')^{-\be_k+[-\mathbf{b}_k']_+}=q^{-2\dt_{i_k}}(\mathbf{x}')^{-\be_k+[-\mathbf{b}_k']_+}(\mathbf{x}')^{-\be_k+[\mathbf{b}_k']_+}.
		\]
		
		By the quantum binomial identity (\cite{Lu10}*{1.3.5}), in $F_{q}(\mathfrak{n}(w))$ we have
		\begin{align*}
		&(x_k'')^l=\sum_{t=0}^l q^{-t(l-t)\dt_{i_k}}\qbinom{l}{t}_{q^{-\dt_{i_k}}}(\mathbf{x}')^{-t\be_k+t[\mathbf{b}_k']_+}(\mathbf{x}')^{-(l-t)\be_k+(l-t)[-\mathbf{b}_k']_+}\\&=\sum_{t=0}^l q^{-t(l-t)\dt_{i_k}+1/2\Lambda'(-t\be_k+t[\mathbf{b}_k']_+,-l\be_k+t[\mathbf{b}_k']_++(l-t)[-\mathbf{b}_k']_+)}\qbinom{l}{t}_{q^{-\dt_{i_k}}}(\mathbf{x}')^{-l\be_k+t[\mathbf{b}_k']_++(l-t)[-\mathbf{b}_k']_+}.
		\end{align*}
		Multiplying $(x_k')^l$ on both sides, we get the following equation in $\intA_{q^{1/2}}(\mathfrak{n})$,
		\begin{align*}
		&(x_k')^l(x_k'')^l\\&=\sum_{t=0}^l q^{-t(l-t)\dt_{i_k}+1/2\Lambda'((l-t)\be_k+t[\mathbf{b}_k']_+,-l\be_k+t[\mathbf{b}_k']_++(l-t)[-\mathbf{b}_k']_+)}\qbinom{l}{t}_{q^{-\dt_{i_k}}}(\mathbf{x}')^{t[\mathbf{b}_k']_++(l-t)[-\mathbf{b}_k']_+}.
		\end{align*}
		Recall the $\mathbb{Z}[q^{\pm1/2}]$-module structure on $R_\varepsilon$ where $q^{1/2}$ acts by $\varepsilon^{1/2}=\varepsilon^{(l+1)/2}$. Since $2\dt_k$ is coprime to the odd number $l$, the number $\varepsilon^{-d_k/2}$ is a primitive $l$-th root of unity. By \cite{Lu10}*{34.1.2}, in $R_\varepsilon$ we have 
		\[
		\qbinom{l}{t}_{q^{-\dt_{i_k}}}=0\qquad\text{unless $l\mid t$.}
		\]
		
		Therefore in $\intA_\varepsilon(\mathfrak{n})$ we have 
		\begin{align}\label{eq:epx}
		{_\varepsilon(x_k')}^l{_\varepsilon(x_k'')}^l={_\varepsilon(\mathbf{x}')}^{l[-\mathbf{b}_k']_+}+{_\varepsilon(\mathbf{x}')}^{l[\mathbf{b}_k']_+}.
		\end{align}
		
		On the other hand, in $\intA_1(\mathfrak{n})$ we have
		\[
		_1(x_k'){_1(x_k'')}={_1(\mathbf{x}')^{[-\mathbf{b}_k']_+}}+{_1(\mathbf{x}')}^{[\mathbf{b}_k']_+}.
		\]
		Applying $\dofr$ and using our assumption, we have
		\[
		{_\varepsilon(x_k')}^l\dofr({_1x_k''})={_\varepsilon(\mathbf{x}')}^{l[-\mathbf{b}_k']_+}+{_\varepsilon(\mathbf{x}')}^{l[\mathbf{b}_k']_+}.
		\]
		Combining with \eqref{eq:epx}, we have
		\begin{equation}\label{eq:zero}
		{_\varepsilon(x_k')}^l(\dofr({_1x_k''})-{_\varepsilon x_k''})=0.
		\end{equation}
		
		Thanks to Corollary \ref{cor:rest}, the element $\dofr({_1x_k''})$ belongs to $\intA_\varepsilon(\mathfrak{n}(w))$. Hence \eqref{eq:zero} holds in $\intA_\varepsilon(\mathfrak{n}(w))$. Since $\intA_\varepsilon(\mathfrak{n}(w))$ is an integral domain by Lemma \ref{le:qeo} (a), we deduce that
		$
		\dofr({_1x_k''})={_\varepsilon x_k''}.
		$
		Therefore \eqref{eq:nufr} is proved.
		
		We next prove $\dfr\big({_\varepsilon(\mathbf{x}'')^{\mathbf{a}}})={_1(\mathbf{x}'')^{\mathbf{a}/l}}$, for $\mathbf{a}=(a_k)_{k\in J}\in\mathbb{N}^J$. Take $\mathbf{a}'=(a_t')_{t\in J}$ such that $\mathbf{a}=\mathbf{a}'+a_k\be_k$. Then $a_k'=0$.
		
		Firstly suppose $l\mid a_k$. In $\intA_{q^{1/2}}(\mathfrak{n})$ one has 
		\[
		(\mathbf{x}'')^{\mathbf{a}}=q^{-1/2\Lambda(a_k\be_k,\mathbf{a})}(x_k'')^{a_k}(\mathbf{x}'')^{\mathbf{a}'}=q^{-1/2\Lambda(a_k\be_k,\mathbf{a})}(x_k'')^{a_k}(\mathbf{x}')^{\mathbf{a}'}.
		\]
		We deduce that 
		\begin{align*}
		\dfr({_\varepsilon(\mathbf{x}'')^\mathbf{a}})&=\varepsilon^{-1/2\Lambda(a_k\be_k,\mathbf{a})}\dfr({_\varepsilon(x_k'')^{a_k}}{_\varepsilon(\mathbf{x}')^{\mathbf{a}'}})\\&=\dfr(\dofr({_1(x_k'')^{a_k/l}}){_\varepsilon(\mathbf{x}'})^{\mathbf{a}'})\\&={_1(x_k'')^{a_k/l}}\dfr({_\varepsilon(\mathbf{x}')^{\mathbf{a}'}})\\&={_1(x_k'')^{a_k/l}}{_1(\mathbf{x}')^{\mathbf{a}'/l}}\\&={_1(x_k'')^{a_k/l}}{_1(\mathbf{x}'')^{\mathbf{a}'/l}}={_1(\mathbf{x}'')^{\mathbf{a}/l}}.
		\end{align*}
		
		Next suppose $l\nmid a_k$. We need to show that $\dfr({_\varepsilon(\mathbf{x}'')^\mathbf{a}})=0$. Notice that in $F_{q}(\mathfrak{n}(w))$ the product $(x_k')^{lc}(\mathbf{x}'')^\mathbf{a}$ ($c\in\mathbb{Z}$) is a $\mathbb{Z}[q^{\pm1/2}]$-linear combination of monomials $(\mathbf{x}')^{(-a_k+lc)\be_k+\mathbf{r}}$ for various $\mathbf{r}=(r_t)_{t\in J}$ with $r_k=0$. Choose $c>0$ such that $-a_k+lc>0$. Then each term $(\mathbf{x}')^{(-a_k+lc)e_k+\mathbf{r}}$ belongs to $\intA_{q^{1/2}}(\mathfrak{n})$. Hence in $\intA_\varepsilon(\mathfrak{n})$ the product $_\varepsilon(x_k')^{lc}{_\varepsilon(\mathbf{x}'')^\mathbf{a}}$ is a $\mathbb{Z}[\varepsilon]$-linear combination of elements of the form  $_\varepsilon(\mathbf{x}')^{(-a_k+lc)\be_k+\mathbf{r}}$, for various $\mathbf{r}$ as before. Since $-a_k+lc$ is not divided by $l$ and $\mathbf{r}$ does not have $\be_k$ factor, we deduce that $\dfr(_\varepsilon(\mathbf{x}')^{(-a_k+lc)\be_k+\mathbf{r}})=0$ by our assumption. Therefore 
		\[
		0=\dfr(_\varepsilon(x_k')^{lc}{_\varepsilon(\mathbf{x}'')^\mathbf{a}})=\dfr(\dofr({_1(x_k')^c}){_\varepsilon(\mathbf{x}'')^\mathbf{a}})={_1(x_k')^c}\dfr({_\varepsilon(\mathbf{x}'')^\mathbf{a}}).
		\]
		Since $\intA_1(\mathfrak{n})$ is an integral domain by Lemma \ref{le:qeo} (a), we deduce that $\dfr({_\varepsilon(\mathbf{x}'')^\mathbf{a}})=0$, which completes the proof.
	\end{proof}
	
	\begin{remark}
		Analogous computations in the induction step also show in  \cites{HLY,NTY}, where the authors used them to construct a canonical central subalgebra and a regular trace map for a root of unity quantum cluster algebra.
	\end{remark}
	
	\section{Frobenius splittings of unipotent subgroups}
	
	Retain the same notations as in the previous section. We further assume that $l=p$ is a prime number. Fix an algebraically closed field $\kk$ with characteristic $p$. 
	
	In this section, a \emph{scheme} will always mean a separated scheme of finite type defined over $\kk$. A \emph{variety} is a reduced scheme. We will not distinguish a variety with its $\kk$-rational points.
	
	\subsection{Frobenius splittings}
	
	Following \cite{BK}, we recall the concept of Frobenius splittings in this subsection.
	
	For a scheme $X$, the \emph{absolute Frobenius morphism}
	\[
	F=F_X:X\longrightarrow X
	\]
	is the identity map on the underlying space, and the $p$-th power map on the structure sheaf $\mathcal{O}_X$.
	
	A \emph{Frobenius splitting} of $X$ is a morphism (as sheaves of abelian groups) 
	\[
	\varphi:\mathcal{O}_X\longrightarrow \mathcal{O}_X
	\]
	such that 
	
	(a) $\varphi(f^pg)=f\varphi(g)$ for any local sections $f,g\in \mathcal{O}_X$, and 
	
	(b) $\varphi(1)=1$.
	
	Note that (a) is equivalent to the requirement that $\varphi\in \text{Hom}_{\mathcal{O}_X}(F_*\mathcal{O}_X,\mathcal{O}_X)$.
	
	Assume $X$ is nonsingular, and let $\omega_X$ be the canonical sheaf of $X$. One has the canonical isomorphism (\cite{BK}*{\S 1.3})
	\begin{equation}\label{eq:Fspl1}
	\text{Hom}_{\mathcal{O}_X}(F_*\mathcal{O}_X,\mathcal{O}_X)\cong H^0(X,\omega_X^{1-p})
	\end{equation}
	as $\mathcal{O}_X(X)$-modules. 
	
	Let $Y$ be a closed subscheme of $X$. A Frobenius splitting $\varphi$ of $X$ is said to \emph{compatibly split} $Y$ if $\varphi$ preserves the ideal sheaf defining $Y$.
	
	Let $R$ be a (finitely generated) $\kk$-algebra and $X=\text{Spec}\,R$. Denote by $\text{End}_F(R)$ the space of additive maps $\varphi:R\rightarrow R$, such that $\varphi(f^pg)=f\varphi(g)$ for any $f,g\in R$. Then one has canonical isomorphisms as $R$-modules,
	\begin{equation}\label{eq:Fspl}
	\text{End}_F(R)\cong \text{Hom}_{\mathcal{O}_X}(F_*\mathcal{O}_X,\mathcal{O}_X).
	\end{equation}
	
	A Frobenius splitting of $X=\text{Spec}\,R$ is equivalent to a map $\varphi\in\text{End}_F(R)$ such that $\varphi(1)=1$. Such a map will also be called a Frobenius splitting of the $\kk$-algebra $R$.
	
	\subsection{Algebraic Frobenius splittings}\label{sec:alg}
	
	We endow $\kk$ with the $\mathbb{Z}[q^{\pm1/2}]$-module structure via $q^{1/2}\mapsto 1$. Fix $w\in W$. Let
	\[
	\intA_\kk(\mathfrak{n}(w))=\kk\otimes_{\mathbb{Z}[q^{\pm 1/2}]}\intA_{q^{1/2}}(\mathfrak{n}(w)).
	\]
	We write $_\kk f=1\otimes f\in \intA_\kk(\mathfrak{n}(w))$, for any $f\in\intA_{q^{1/2}}(\mathfrak{n}(w))$.
	
	%Retain the same notation as in \S \ref{sec:sec2}. Further assume $l=p$. 
	
	Recall that $\varepsilon$ is a primitive $p$-th root of unity. We further endow $\kk$ with the $\mathbb{Z}[\varepsilon]$-algebra structure via $\varepsilon\mapsto 1$. Then we have canonical isomorphisms as $\kk$-algebras
	\begin{align*}
	\intA_\kk(\mathfrak{n}(w))\cong \kk\otimes_{\mathbb{Z}[\varepsilon]}\intA_\varepsilon(\mathfrak{n}(w))\cong \kk\otimes_{\mathbb{Z}[\varepsilon]}\intA_1(\mathfrak{n}(w)).
	\end{align*}
	
	Suppse $l(w)=r$. Fix a reduced expression $\underline{w}=(i_r,\dots,i_1)$ of $w$. Recall from \S \ref{sec:PBW} the dual PBW basis $\{e_{\underline{w}}^*(\mathbf{n})\mid\mathbf{n}\in\mathbb{N}^r\}$ of $\intA_{q^{1/2}}(\mathfrak{n}(w))$. Write $y_t\in\intA_\kk(\mathfrak{n}(w))$ to be the image of dual PBW basis element corresponding to $\mathbf{n}=\be_t$ for $1\leq t\leq r$. Here $\be_t$ ($1\leq t\leq r$) is the standard basis element of $\mathbb{Z}^r$ as usual.
	
	We endow $\intA_\kk(\mathfrak{n}(w))$ with a $Q_+$-grading given by setting $\wt(y_t)=\beta_t$ where $\beta_t=s_{i_1}\dots s_{i_{t-1}}(\alpha_{i_t})\in Q_+$ for $1\leq t\leq r$. The degree of a homogeneous element in $\intA_\kk(\mathfrak{n}(w))$ will be called its weight as before.
	
	The following claim follows from the dual Levendorskii--Soibelman formula \eqref{eq:prog} and the fact that $\{e^*_{\underline{w}}(\mathbf{n})\mid\mathbf{n}\in\mathbb{N}^r\}$ is a basis.
	
	(a) {\em Elements $y_t$ ($1\leq t\leq r$) are algebraically independent in $\intA_\kk(\mathfrak{n}(w))$. Moreover these elements generate $\intA_\kk(\mathfrak{n}(w))$ as a $\kk$-algebra.}
	
	In other words one has
	\begin{equation}\label{eq:poly}
	\intA_\kk(\mathfrak{n}(w))=\kk[y_1,\dots,y_t]
	\end{equation}
	as $\kk$-algebras. The identity \eqref{eq:poly} gives an isomorphism as varieties 
	\begin{equation}\label{eq:poly1}
	    \text{Spec }\intA_\kk(\mathfrak{n}(w))\cong \mathbb{A}^r.
	\end{equation}
	This isomorphism depends on the choice of the reduced expressions $\underline{w}$.
	
	Recall from Corollary \ref{cor:rest} that $\dofr$ and $\dfr$ restrict to subalgebras, that is,
	\[
	\dofr\mid_{\intA_1(\mathfrak{n}(w))}:\intA_1(\mathfrak{n}(w))\rightarrow\intA_\varepsilon(\mathfrak{n}(w)),\quad\dfr\mid_{\intA_\varepsilon(\mathfrak{n}(w))}:\intA_\varepsilon(\mathfrak{n}(w))\rightarrow\intA_1(\mathfrak{n}(w)).
	\]
	
	Set
	\begin{align*}
	&F_w=(\cdot)^p\otimes_{\mathbb{Z}[\varepsilon]}\dofr\mid_{\intA_1(\mathfrak{n}(w))}:\intA_\kk(\mathfrak{n}(w))\longrightarrow\intA_\kk(\mathfrak{n}(w)),\\
	&\varphi_w=(\cdot)^{1/p}\otimes_{\mathbb{Z}[\varepsilon]}\dfr\mid_{\intA_\varepsilon(\mathfrak{n}(w))}:\intA_\kk(\mathfrak{n}(w))\longrightarrow\intA_\kk(\mathfrak{n}(w)),
	\end{align*}
	where $(\cdot)^p$ and $(\cdot)^{1/p}$ are automorphisms on $\kk$ given by $t\mapsto t^p$ and $t\mapsto t^{1/p}$, respectively.
	
	\begin{prop}\label{prop:frclu}
		The map $F_w$ is the $p$-th power map on $\intA_\kk(\mathfrak{n}(w))$, and the map $\varphi_w$ is a Frobenius splitting of the $\kk$-algebra $\intA_\kk(\mathfrak{n}(w))$.
		
		Moreover, for any quantum cluster monomial $\mathbf{x}^{\mathbf{a}}\in \intA_{q^{1/2}}(\mathfrak{n}(w))$ ($\mathbf{a}\in\mathbb{N}^{r}$), we have
		\[ \varphi_w({_\kk\mathbf{x}^{\mathbf{a}}})={_\kk\mathbf{x}^{\mathbf{a}/p}}.
		\]
	\end{prop}
	
	\begin{proof}
		The assertion that $F_w$ is the $p$-th power map follows from the similar argument as in \cite{BS}*{Proposition 5.2 (1)}. The fact that $\varphi_w$ is a Frobenius splitting follows from Proposition \ref{prop:Frli}. The remaining part is a direct consequence of Theorem \ref{thm:main}.
	\end{proof}
	
	The Frobenius splitting $\varphi_w$ will be called the \emph{algebraic splitting} of $\intA_\kk(\mathfrak{n}(w))$. The next two subsections are devoted to show that the algebraic splittings coincide with the canonical splittings on Schubert cells in finite types. 
	
	\subsection{Geometric description of the algebraic splittings}\label{sec:homog}
	
	We write $N(w)=\text{Spec}\,\intA_\kk(\mathfrak{n}(w))$. Then the map $\varphi_w$ gives a Frobenius splitting of $N(w)$. The goal of this subsection is to describe the algebraic splitting $\varphi_w$ under the isomorphism \eqref{eq:Fspl1}.
	
	Thanks to \eqref{eq:poly1}, the canonical sheaf of $N(w)$ is trivial. Take a nowhere vanishing global section $v\in H^0(N(w),\omega_{N(w)}).$
	Note that the choice of $v$ is unique up to $\kk^*$-rescaling. 
	
	By \eqref{eq:Fspl} we have isomorphisms
	\begin{equation}\label{eq:isoF}
	\text{End}_F(\intA_\kk(\mathfrak{n}(w)))\cong H^0(N(w),\omega_{N(w)}^{1-p})=\intA_\kk(\mathfrak{n}(w))v^{1-p}
	\end{equation}
	as $\intA_\kk(\mathfrak{n}(w))$-modules. 
	
	Set $$\text{Supp}(w)=\{i\in I\mid s_i\leq w\text{ in the Bruhat order}\}$$ to be the set of indices of the support of $w$. For $i\in \text{Supp}(w)$, the quantum unipotent minor $D(w^{-1}\varpi_i,\varpi_i)$ is a frozen cluster variable in $\intA_{q^{1/2}}(\mathfrak{n}(w))$ up to a power of $q$. Therefore the element $$p_i={_\kk D(w^{-1}\varpi_i,\varpi_i)}\in\intA_\kk(\mathfrak{n}(w))$$ is a prime element in $\intA_\kk(\mathfrak{n}(w))$. Set 
	\[
	Y_i=\{x\in N(w)\mid {p_i}(x)=0\}
	\]
	to be the prime divisor of $N(w)$ defined by $p_i$. It corresponds to the Richardson divisor of the Schubert cell.
	
	\begin{prop}\label{prop:comp}
		The splitting $\varphi_w$ of $N(w)$ compatibly splits divisors $Y_i$, for $i\in\text{Supp}(w)$.
	\end{prop}
	
	\begin{proof}
		Take $(\mathbf{x}_t)_{t\in J}$ to be a cluster of $\intA_{q^{1/2}}(\mathfrak{n}(w))$. Let $T$ be the $\kk$-subalgebra of the fractional field of $\intA_\kk(\mathfrak{n}(w))$, generated by ${_\kk x_t}$ $(t\in J_\fz)$ and ${_\kk (x_{t'}^{\pm1})}$ $(t'\in J_\ex)$. By the Laurent phenomenon, the algebra $\intA_\kk(\mathfrak{n}(w))$ is contained in $T$. Let $U=\text{Spec }T$. Then $U$ is an open dense subset of $N(w)$. The Frobenius splitting $\varphi_w$ restricts a splitting of $U$. Thanks to Proposition \ref{prop:frclu}, this splitting compatibly splits subvarieties (of $U$) defined by $_\kk x_t\mid_U=0$ for any $t\in J_\fz$. By construction we have
		\[
		\prod_{t\in J_\fz}{_\kk x_t}=\prod_{i\in\text{Supp}(w)}p_i.
		\]
		Hence ${\varphi_w\mid_U}$ compatibly splits $Y_i\cap U$ for any $i\in\text{Supp}(w)$. By \cite{BK}*{Lemma 1.1.7 (ii)} we complete the proof.
	\end{proof}
	
	% by 
	% \[
	% (t\otimes \varepsilon)\cdot f=t^{\langle \varepsilon,\mu\rangle} f
	% \]
	% for $f\in\intA_\kk(\mathfrak{n})$ with $\wt(f)=\mu$.
	
	Let $\kk[P]$ be the group algebra (over $\kk$) of the weight lattice $P$, and $H$ be the algebraic torus associated to $\kk[P]$. The $Q_+$-grading on $\intA_\kk(\mathfrak{n}(w))$ gives an algebraic action of the torus $H$ on $N(w)$. Then $H$ also acts on the space $\text{End}_F(\intA_\kk(\mathfrak{n}(w))$ as $\kk$-linear maps, given by
	\[
	(h*\varphi)f=h(\varphi(h^{-1}f))\qquad\text{for $h\in H$, $\varphi\in\intA_\kk(\mathfrak{n}(w))$ and $f\in \intA_\kk(\mathfrak{n}(w))$}.
	\]
	A splitting $\varphi$ of $\intA_\kk(\mathfrak{n}(w))$ is called \emph{H-equivariant} if $h*\varphi=\varphi$, for all $h\in H$. It is easy to see that a splitting $\varphi$ is $H$-equivariant if and only if
	\[
	\varphi\big(\intA_\kk(\mathfrak{n}(w))_\mu\big)\subset \intA_\kk(\mathfrak{n}(w))_{\mu/p}\quad \text{for $\mu\in Q_+$.}
	\]
	Here $\intA_\kk(\mathfrak{n}(w))_{\mu/p}$ is understood as zero space if $\mu\not\in pQ_+$. 
	
	\begin{prop}\label{prop:char}
		The splitting $\varphi_w$ is the unique $H$-equivariant splitting which compatibly splits all the divisors $Y_i$ ($i\in \text{Supp}(w)$). 
		
		Moreover, this unique splitting is given by $\sigma_w^{p-1}$ under the isomorphism \eqref{eq:isoF} up to a $\kk^*$-scalar, where 
		\[
		\sigma_w=(\prod_{i\in\text{Supp}(w)}p_i)v^{-1}\in H^0(N(w),\omega_{N(w)}^{-1}).
		\]
	\end{prop}
	
	\begin{proof}
		Thanks to Proposition \ref{prop:comp} and \eqref{eq:homo}, the splitting $\varphi_w$ satisfies the requirement. We next prove the uniqueness. 
		
		Let $\varphi$ be an $H$-equivariant Frobenius splitting which compatibly splits $Y_i$, for all $i\in\text{Supp}(w)$. Suppose $\varphi$ is given by $gv^{1-p}\in H^0(N(w),\omega_{N(w)}^{1-p})$ for some $g\in \intA_\kk(\mathfrak{n}(w))$. By \cite{BK}*{Lemma 4.1.14}, the isomorphism \eqref{eq:isoF} is moreover $H$-equivariant. Since $\varphi$ is $H$-equivariant and $v^{1-p}$ is an $H$-eigenvector, we deduce that $g$ is an $H$-eigenvector. Equivalently, the element $g$ is homogeneous with respect to the $Q_+$-grading. 
		
		Take a reduced expression $\underline{w}=(i_r,\dots,i_1)$ of $w$. Set $\beta_k=s_{i_1}\dots s_{i_{k-1}}(\alpha_{i_k})\in Q_+$ for $1\leq k\leq r$. Recall the elements $y_t$ ($1\leq t\leq r$) and the equality \eqref{eq:poly} in \S \ref{sec:alg}.
		
		We write $g$ in terms of the polynomial on variables $y_1,\dots, y_r$. Then the monomial $y_1^{p-1}\dots y_r^{p-1}$ occurs with a nonzero coefficient (\cite{BK}*{Example 1.3.1}). Hence
		\begin{equation}\label{eq:wtcon}
		\wt(g)=\wt(y_1^{p-1}\dots y_r^{p-1})=(p-1)(\beta_1+\dots+\beta_r).
		\end{equation}
		
		On the other hand, take any $i\in \text{Supp}(w)$. We 
		claim that $p_i^{p-1}$ divides $g$. Suppose the claim is not correct. Write $g=p_i^lg_1$, with $0\leq l<p-1$ and $(p_i,g_1)=1$. 
		
		Take $x\in Y_i$ to be a smooth point of $Y_i$. Then we can choose a system of local coordinates $t_1,\dots,t_r$ at $x$ with $t_1=p_i$ in the local ring $\mathcal{O}_{N(w),x}$. Since $g_1(x)\neq 0$, the section $g_1v^{1-p}$ does not vanish at $x$, which implies the local expression of $g_1v^{1-p}$ have nonzero constant term. Therefore at the point $x$, the section $gv^{1-p}=t_1^lg_1v^{1-p}$ has the local expression of the form
		\[
		t_1^l(a+\sum_{\mathbf{j}\neq\mathbf{0}}d_{\mathbf{j}}\dt^\mathbf{j})(dt_1\wedge\dots\wedge dt_r)^{1-p}\qquad \text{with $a\neq 0$}.
		\]

		By \cite{BK}*{Lemma 1.3.6 \& Proposition 1.3.7}, the splitting $\varphi_w$ induces a splitting on the local ring $\mathcal{O}_{N(w),x}\subset \kk[[t_1,\dots,t_r]]$ given by
		\[
		\widetilde{\varphi_w}(f)={\text{Tr}}\big(ft_1^l(a+\sum_{\mathbf{j}\neq\mathbf{0}}d_{\mathbf{j}}\dt^\mathbf{j})\big)\qquad \text{for $f\in\mathcal{O}_{N(w),x}$.}
		\]
		Here ${\text{Tr}}$ is the trace map on the ring of formal power series defined as in \cite{BK}*{Lemma 1.3.6}. Moreover $\widetilde{\varphi_w}$ should preserves the principal ideal generated by $t_1$, because $\varphi_w$ compatibly splits $Y_i$. 
		
		Note that the monomial $t_1^{p-1-l}t_2^{p-1}\dots t_r^{p-1}$ belongs to $t_1\mathcal{O}_{N(w),x}$, but its image  $$\widetilde{\varphi_w}(t_1^{p-1-l}t_2^{p-1}\dots t_r^{p-1})\in \kk[[t_1,\dots, t_r]]$$ has constant term $a$, which implies that it does not belong to $t_1\mathcal{O}_{N(w),x}$. This is a contradiction.
		
		We have proved that $p_i^{p-1}\mid g$ for any $i\in\text{Supp}(w)$. Hence we can write 
		$$g=\prod_{i\in\text{Supp}(w)}p_i^{p-1}g',$$
		for some $g'\in\intA_\kk(\mathfrak{n}(w))$.
		
		Notice that 
		\[
		\wt\big(\prod_{i\in\text{Supp}(w)}p_i^{p-1}\big)=(p-1)\big(\sum_{i\in\text{Supp}(w)}\varpi_i-w^{-1}(\sum_{i\in \text{Supp}(w)}\varpi_i)\big)=(p-1)(\beta_1+\dots+\beta_r),
		\]
		where the last equality follows from \cite{KU}*{1.3.22}. Combining  \eqref{eq:wtcon} we deduce that $\wt(g')=0$. Hence, $g'$ belongs to $\kk^*$.
		
		Therefore $\varphi$ is unique and it is determined by $\sigma_w^{p-1}$ up to a nonzero scalar. We complete the proof.
	\end{proof}
	
	\subsection{Comparison with geometric construction}\label{sec:compa}
	
	From now on, we assume that our root datum is of finite type. 
	
	Let $G$ be the reductive group defined over $\kk$ associated to the root datum, with a standard Borel subgroup $B$ and a standard opposite Borel subgroup $B_-$. Let $N$ and $N_-$ be the unipotent radicals of $B$ and $B_-$, respectively. Set $H=B\cap B_-$ to be the maximal torus, and $W=N(H)/H$ to be the Weyl group. For any $w\in W$, set
	\[
	N(w)= N\cap w^{-1}N_-w.
	\]
	It is a closed subgroup of the unipotent group $N$. Let $H$ acts on $N(w)$ by conjugation. Note that one has $N=N(w_0)$, where $w_0\in W$ is the longest element in the Weyl group.
	
	The coordinate ring of $H$ is canonically isomorphic to the group algebra $\kk[P]$, and the coordinate ring of the closed subgroup $N\cap w^{-1}N_-w$ is canonically isomorphic to $\intA_\kk(\mathfrak{n}(w))$ as $\kk$-algebras by \cite{Ki}*{Theorem 4.44}. Moreover, the $H$-action on $N(w)$ by conjugation coincides with the one defined by the $Q_+$-grading as in \S \ref{sec:homog}. Therefore there is no conflict with previous notations.
	
	We notice that although the result in \cite{Ki} is over the field of complex numbers, the proof applies to arbitrary fields, because the author essentially works over integral forms. 
	
	Let $G/B$ be the associated flag variety and $BwB/B$ be the Schubert cell for $w\in W$. Let $H$ acts on Schubert cells by left multiplications. Then there is an $H$-equivariant isomorphism, 
	\begin{equation}\label{id}
	\iota_w:N(w)\xrightarrow{\sim} Bw^{-1}B/B,\qquad\text{given by } n\mapsto n\cdot w^{-1}B/B.
	\end{equation}
	
	By \cite{BK}*{Theorem 4.1.15} the flag variety $G/B$ admits a \emph{$B$-canonical} (see \cite{BK}*{\S 4.1} for the definition) Frobenius splitting $\varphi$ which compatibly splits Schubert varieties $\overline{BwB/B}$ and opposite Schubert varieties $\overline{B_-wB/B}$, for any $w\in W$. Therefore the splitting $\varphi$ induces a splitting of the Schubert cell $Bw^{-1}B/B$, which we call the \emph{canonical splitting}. Under the isomorphism $\iota_w$, we get a splitting $\varphi'_w$ of $N(w)$.
	
	\begin{corollary}\label{cor:alggeo}
		As Frobenius splittings of $N(w)$, we have $\varphi_w=\varphi_w'$.
	\end{corollary}
	
	\begin{proof}
		For $\lambda\in P_+$, set $\intV_\kk(\lambda)=\kk\otimes_{\mathbb{Z}[q^{\pm1}]}\intV(\lambda)$. Then $\intV_\kk(\lambda)$ carries a rational $G$-action. The bilinear form $_\kk(\cdot,\cdot)_\lambda$ on $\intV_\kk(\lambda)$ is defined to be the base change of the form defined in \S \ref{sec:minor}. Then for any $i\in I$ and $x\in N(w)$ we have
		\begin{align*}
		_\kk(x v_{w^{-1}\varpi_i},v_{\varpi_i})_{\varpi_i}=0 &\iff {_\kk(x\dot{w}^{-1} v_{\varpi_i},v_{\varpi_i})_{\varpi_i}}=0\\&\iff x\dot{w}^{-1}\in \overline{B_-s_iB}\\&\iff \iota_w(x)\in \overline{B_-s_iB/B}, 
		\end{align*}
		where $\dot{w}\in G$ is a lift of the Weyl group element $w$.
		
		Therefore we have
		\begin{equation}\label{eq:ioY}
		\iota_w(Y_i)=Bw^{-1}B/B\cap \overline{B_-s_iB/B},\quad \text{for $i\in \text{Supp}(w)$.}
		\end{equation}
		
		By \cite{BK}*{Lemma 1.1.7}, the induced splitting of $Bw^{-1}B/B$ compatibly splits the intersection $Bw^{-1}B/B\cap \overline{B_-s_iB/B}$, so the splitting $\varphi_w'$ compatibly splits $Y_i$ for any $i\in \text{Supp}(w)$ by \eqref{eq:ioY}. Since $\varphi$ is $B$-canonical and $\iota_w$ is $H$-equivariant, we deduce that $\varphi_w'$ is $H$-equivariant. Hence by Proposition \ref{prop:char} we have $\varphi_w=\varphi_w'$.
	\end{proof}
	
	% \begin{prop}
	%     Under the isomorphism \eqref{eq:Fspl}, the splitting $\varphi_w$ is given by the element 
	%     $$\big({_\kk D(w^{-1}\rho_w,\rho_w)}\big)^{p-1}={_\kk D(w^{-1}(p-1)\rho_w,(p-1)\rho_w)}.$$  Therefore, $\varphi_w$ is the unique $(p-1)$-th power splitting of $N(w)$ whose compatibly split prime divisors are exactly $Y_i$, for all $i\in I$ such that $s_i$ belongs to the support of $w$. 
	% \end{prop}
	
	\subsection{Relation with reduction maps}\label{sec:red}
	
	In this subsection we explain that the canonical splittings of Schubert cells are compatible with reduction maps. 
	
	Suppose $w=v'v$ in $W$ with $l(w)=l(v')+l(v)$. We choose a reduced expression $\underline{w}=(i_{l(w)},\dots,i_1)$ such that $(i_{l(v)}\dots,i_1)$ gives a reduced expression of $v$. It follows from the definitions that we have a natural embedding as $\kk$-algebras,
	\begin{equation*}
	\begin{tikzcd}
	&\intA_{\kk}(\mathfrak{n}(v))\arrow[r,hook]&\intA_{\kk}(\mathfrak{n}(w)).
	\end{tikzcd}
	\end{equation*}
	
	We prove Corollary \ref{cor:red} here. By \eqref{eq:poly1}, take isomrphisms $N(w)\cong \mathbb{A}^{l(w)}$ and $N(v)\cong\mathbb{A}^{l(v)}$ associated to the chosen reduced expressions. It is direct to see that under the isomorphisms $\iota_w$ and $\iota_v$, the reduction map is just the projection onto the first $l(v)$ coordinates. Therefore the comorphism $(\pi_v^w)^*:\intA_{\kk}(\mathfrak{n}(v))\rightarrow\intA_{\kk}(\mathfrak{n}(w))$ coincides with the natural embedding. 
		
		Hence it will suffice to check the commutativity of the following diagram as abelian groups,
		\begin{equation*}
		\begin{tikzcd}
		&\intA_\kk(\mathfrak{n}(v))\arrow[r,hook] \arrow[d,"\varphi_v"'] &\intA_\kk(\mathfrak{n}(w)) \arrow[d,"\varphi_w"]\\& \intA_\kk(\mathfrak{n}(v))\arrow[r,hook]&\intA_\kk(\mathfrak{n}(w)).
		\end{tikzcd}
		\end{equation*}
		This diagram clearly commutes, because by construction $\varphi_v$ and $\varphi_w$ are both restrictions of the same map on $\intA_\kk(\mathfrak{n})$.

\end{document}